\RequirePackage{fix-cm}
\documentclass[smallextended]{svjour3}       
\smartqed  
\usepackage[english]{babel}
\usepackage{amsmath}
\usepackage{amsfonts}
\usepackage{amssymb}
\usepackage{mathtools}
\usepackage{nccmath}

\usepackage[T1]{fontenc} 
\usepackage{fancybox} 
\usepackage{alltt}
\usepackage{graphicx,epstopdf}
\usepackage{float}
\usepackage{epsfig}
\usepackage{fullpage}
\usepackage{fancyhdr}
\usepackage{moreverb}
\usepackage{xspace}
\usepackage{multirow}

\usepackage[colorlinks,hyperindex,bookmarks=false,linkcolor=blue,citecolor=blue,urlcolor=blue]{hyperref}
\usepackage{wrapfig}
\usepackage{epsf}
\usepackage{bm}
\usepackage{textgreek}
\usepackage[skip=1pt]{caption}
\usepackage[skip=1pt]{subcaption}
\usepackage[ruled,vlined,norelsize]{algorithm2e}	
\usepackage{psfrag}
\usepackage{cite}
\usepackage[utf8]{inputenc}

\newtheorem{Theorem}{Theorem}
\newtheorem{Cor}{Corollary}
\newtheorem{Prop}{Proposition}
\newtheorem{Definition}{Definition}
\newtheorem{num_example}{Example}
\newcommand{\MATLAB}{\textsc{Matlab}~}
\newcommand{\LL}{\mathbb{L}}
\newcommand{\sL}{\mathbb{S}}
\newcommand{\CC}{\mathbb{C}}

\newcommand{\by}{{\bf y}}
\newcommand{\bfz}{{\bf 0}}

\newcommand{\bX}{{\bf X}}

\newcommand{\bw}{{\bf w}}
\newcommand{\bv}{{\bf v}}
\newcommand{\bA}{{\bf A}}
\newcommand{\bB}{{\bf B}}
\newcommand{\bC}{{\bf C}}
\newcommand{\bD}{{\bf D}}

\newcommand{\bH}{{\bf H}}
\newcommand{\bE}{{\bf E}}
\newcommand{\bx}{{\bf x}}
\newcommand{\bM}{{\bf M}}

\newcommand{\bI}{{\bf I}}
\newcommand{\bP}{{\bf P}}
\newcommand{\bR}{{\bf R}}
\newcommand{\bL}{{\bf L}}

\newcommand{\bV}{{\bf V}}
\newcommand{\bW}{{\bf W}}
\newcommand{\bY}{{\bf Y}}

\newcommand{\br}{{\bf r}}

\newcommand{\ls}{\mathrm{s}}
\newcommand{\ii}{\mathrm{i}}
\newcommand{\C}{\mathcal{C}}

\begin{document}

\title{{An efficient, memory-saving approach for the Loewner framework}}
\author{Davide Palitta        \and
       Sanda Lefteriu
}
\institute{D. Palitta \at
              Max Planck Institute for Dynamics of Complex Technical Systems\\
Sandtorstraße 1, 39106 Magdeburg, Germany\\
              \email{palitta@mpi-magdeburg.mpg.de}           
           \and
           S. Lefteriu \at
            IMT Lille Douai, Institut Mines-Télécom, Univ. Lille, Centre for Education, Research and Innovation (CERI) Digital
Systems,\\
F-59000 Lille, France\\
\email{sanda.lefteriu@imt-lille-douai.fr} 
}
\date{Received: date / Accepted: date}
\maketitle

\begin{abstract}
The Loewner framework is one of the most successful data-driven model order reduction techniques.
If $N$ is the cardinality of a given data set, the so-called Loewner and shifted Loewner matrices $\LL\in\mathbb{C}^{N\times N}$ and $\sL\in\mathbb{C}^{N\times N}$ can be defined by solely relying on information encoded in the considered data set and they play a crucial role in the computation of the sought rational model approximation.
 In particular, the singular value decomposition of a linear combination of $\sL$ and $\LL$ provides the tools needed to construct accurate models which fulfill important approximation properties with respect to the original data set. 
 However, for highly-sampled data sets, the dense nature of $\LL$ and $\sL$ leads to numerical difficulties, namely the failure to allocate these 
matrices in certain memory-limited environments or excessive computational costs. Even though they do not possess any sparsity pattern, the 
Loewner and shifted Loewner matrices are extremely structured and,
in this paper, we show how to fully exploit their Cauchy-like structure  
to reduce the cost of computing accurate rational models while avoiding the explicit allocation of $\LL$ and $\sL$. In particular, the use of the \emph{hierarchically semiseparable} format allows us to remarkably lower both the computational cost and the memory requirements of the Loewner framework obtaining a novel scheme whose costs scale with $N \log N$. 

\end{abstract}

\section{Introduction}
The Loewner framework, originally proposed in \cite{artAJMACA} for solving the generalized realization problem coupled with tangential interpolation, was successfully employed for data-driven model order reduction from frequency domain data \cite{SLACATCAD09}. Measurements of the frequency response are available in several communities: electrical engineering (impedance, admittance or scattering parameters \cite{SLACATCAD09}), mechanical and civil engineering (structural and vibro-acoustic frequency response functions \cite{PolyMAX} or frequency response measurements of thermal systems \cite{Thermalsys}), to name a few. The first step in the Loewner framework consists in setting up the data matrices and building the Loewner and shifted Loewner matrices entry-wise based on the chosen partition into right and left data, followed by computing the singular value decomposition (SVD) of a linear combination of these matrices and forming the model by projection, using the dominant singular triplets. The main advantages of the Loewner framework over existing approaches are, on the one hand, its system identification capabilities, in the sense that the order of the system can be deduced from the singular value drop, and, on the other hand, its potential in dealing with systems with a large number of inputs and outputs efficiently, thanks to incorporating the concept of tangential interpolation. The main drawbacks, however, are the large storage requirements paired with the significant CPU cost inherent to the full SVD computation for data sets with a large number of measurements (values in the range $10^5$ are common in industrial applications). To bypass these inconveniences, greedy-type approaches were proposed in \cite{SLACATCAD09}, thus reducing memory requirements, from $\mathcal{O}(N^2)$ for storing the dense Loewner and shifted Loewner matrices to $\mathcal{O}(N+n^2)$, and the computational cost, from $\mathcal{O}(N^3)$ for computing the SVD to  $\mathcal{O}(Nn^3)$ and $\mathcal{O}(Nn^4)$, where $N$ is the size of the data set and $n$ is the order of the model. 

Taking advantage of numerical linear algebra tools to reduce storage and computational requirements for the Loewner framework is another avenue worth exploring due to the inherent structure embedded in the albeit dense Loewner and shifted Loewner matrices. The factored ADI-Galerkin method for computing these matrices as solutions to certain Sylvester equations with a factored right-hand side was investigated in~\cite{FSVDL}. Such a scheme computes low-rank approximations to the dense Loewner matrix to speed-up the SVD computation. However, in~\cite{FSVDL} no results about the accuracy of the computed reduced models are reported. Moreover, the memory constraints coming from the allocation of $\LL$ and $\sL$ are still present. Alternatively, one can focus on accelerating solely the step of the SVD calculation by employing Krylov methods (see, e.g., \cite{Sto12,Hoc01,BagRei05,Lar98} to name a few), by using the randomized SVD \cite{nakatsukasa2020fast} to compute the dominant singular triplets instead of the full SVD or other types of inexact SVD-type decompositions (adaptive cross approximation \cite{ACA}, particularly suited for hierarchical matrices, or a CUR decomposition \cite{CUR} as in \cite{LoewCUR, 9073015}). 

The novel approach proposed in this paper tackles the issue of the memory requirements, at the same time as reducing the CPU cost of the Loewner framework while maintaining the accuracy of the standard approach for large values of the number of measurements. As the Loewner and shifted Loewner matrices satisfy Sylvester equations with diagonal coefficient matrices, they are, in fact, Cauchy-like matrices, obtained as the Hadamard product between a Cauchy matrix $\C$ and low-rank right-hand sides. Extensive research has been devoted to fully exploiting the rich structure of Cauchy matrices. Several algorithms for computing the matrix-vector product $\C \bx$ can be found in the literature and many avoid assembling the full matrix $\C$ (see, e.g., \cite{Pan2014,GREENGARD1987325,Gohberg1994,Carrier1988}). Hierarchically semiseparable matrices (HSS) have deemed efficient for approximating Cauchy matrices with a low off-diagonal rank \cite{PAN2015107,Pan2014}. HSS and other rank-structured matrices are widely used in developing fast algorithms for algebraic operations (matrix-vector multiplications, matrix factorizations, matrix inversion, etc., see, e.g., \cite{Vandebril2005,Pan2014,PAN2015107,Chandrasekaran2006,Xia2010} and references therein) used as building blocks for the solution of certain problems like linear systems of equations~\cite{XiaCGL10}, eigenvalue problems~\cite{VogXCB16}, linear and quadratic matrix equations~\cite{MassPR18,Kressner2019}, and many more. For our application, the approximation of the Cauchy matrix in HSS format considerably decreases the computational cost of matrix-vector products involving a linear combination of the Loewner and shifted Loewner matrices needed for the partial SVD computation, while avoiding to form them. All results involving HSS-matrices presented in this paper have been obtained by means of the {\tt hm-toolbox} \cite{hmtoolbox}.

The employment of an HSS-representation of $\C$ may introduce some inexactness in our scheme and this has to be taken into account in the iterative SVD computation.
The use of inexact matrix-vector products within iterative procedures has been the subject of numerous research papers: Krylov techniques for solving linear systems and matrix equations \cite{Simoncini2003,Bouras2005,Eshof2004,Kuerschner2019,Kuerschner2018}, eigenvalue problems \cite{Freitag2007,Simoncini2002}, or an inexact variant of the Lanczos bidiagonalization for the computation of leading singular triplets of a generic matrix function \cite{GAAF2017}. In our case, we do not need an accurate approximation of the singular triplets, but rather have meaningful spaces spanned by the computed left and right singular vectors so that the obtained reduced model inherits the desired approximation properties (see, e.g., \cite{IonitaPhD,TutorialLoewner}). 

The remainder of the paper is structured as follows. Section \ref{sect:review} provides a review of the Loewner framework, whereas section \ref{Exploiting the structure of L and S} presents results showcasing the special structure of the Loewner and shifted Loewner matrices as Cauchy-like matrices and their approximation as hierarchically semiseparable matrices allowing for efficient, inexact matrix-vector products in the partial SVD computation. Section \ref{Numerical results} presents the results of our numerical experiments and section \ref{sect:concl} concludes the paper. 

\section{Review of the Loewner framework}\label{sect:review}
The Loewner framework has been proposed to address the rational interpolation/approximation problem. In the control community, this is referred to as  system identification from frequency domain measurements and is stated below.  

\noindent
\begin{problem}[Rational approximation]\label{approx_pr}
Given pairs of points representing the frequency $f_j$, and the corresponding transfer function measurement at that frequency $\bH_j \in \CC^{p \times q}$ for a system with $q$ inputs and $p$ outputs:
\vspace{-.5mm}
\begin{equation} \label{eq:meas}
(f_j;\bH_j),~~j=1,\ldots,N,
\vspace{-.5mm}
\end{equation}
with $p$ and $q$ assumed to be much smaller than $N$, the problem amounts to finding the rational transfer function $\bH(\ls)$  which approximates the data:
 \begin{equation}\label{eq:ss}
 \bH_j \approx \bH(\ls=\ii \omega_j),~~\forall j=1,\ldots,N.
 \end{equation}
Thus, the transfer function  evaluated for the Laplace variable $\ls=\ii \omega_j$, where $\ii^2=-1$ and $\omega_j=2 \pi f_j$, should be close (in some norm) to the corresponding measurement $\bH_j$. Several equivalent representations are possible for the rational transfer function, namely pole-residue, pole-zero, state-space or descriptor-form.
\end{problem}

Most systems of interest are real, with their transfer function satisfying the complex conjugate condition $\bH(\bar \ls)=\overline{\bH(\ls)}$. Hence, we assume that the given data set satisfies this condition and is of the following form:
\begin{equation} \label{eq:meascc}
\left(\ii \omega_j,-\ii \omega_j;\bH_j,\overline{\bH}_j\right),~~j=1,\ldots,N.
\vspace{-.5mm}
\end{equation}

We proceed by presenting the Loewner framework as a solution scheme addressing the rational approximation Problem \ref{approx_pr}. The first step in the Loewner framework \cite{artAJMACA,SLACATCAD09} is partitioning the data in two disjoint sets. This partition influences the conditioning of the problem \cite[Ch. 2.1]{IonitaPhD} and finding the optimal partition for each data set is beyond the scope of this paper. The most natural partitions are summarized in the following (assuming an even number of measurements $N$ and frequencies sorted in ascending order): 
\begin{itemize}
    \item {\sc Half\&Half}: the first half of the data in one set and the other half in the second set:
\begin{equation} \label{eq:partitionfHH}
\left\{\ii \omega_{1},-\ii \omega_{1},\ldots,\ii \omega_{\frac{N}{2}},-\ii \omega_{\frac{N}{2}}\right\} \cup \left\{\ii \omega_{\frac{N}{2}+1},-\ii \omega_{\frac{N}{2}+1},\ldots,\ii \omega_{N},-\ii \omega_{N}\right\}
\end{equation}
and, correspondingly, 
\begin{equation} \label{eq:partitionmHH}
\left\{\bH_{1},\overline{\bH}_{1},\ldots,\bH_{\frac{N}{2}},\overline{\bH}_{\frac{N}{2}}\right\} \cup \left\{\bH_{\frac{N}{2}+1},\overline{\bH}_{\frac{N}{2}+1},\ldots,\bH_{N},\overline{\bH}_{N}\right\},
\end{equation}
     \item {\sc Odd\&Even}: data with odd indices in the first set and data with even indices in the second set:
\begin{equation} \label{eq:partitionfOE}
\left\{\ii \omega_{1},-\ii \omega_{1},\ldots,\ii \omega_{N-1},-\ii \omega_{N-1}\right\} \cup \left\{\ii \omega_{2},-\ii \omega_{2},\ldots,\ii \omega_{N},-\ii \omega_{N}\right\}
\end{equation}
and, correspondingly, 
\begin{equation} \label{eq:partitionmOE}
\left\{\bH_{1},\overline{\bH}_{1},\ldots,\bH_{N-1},\overline{\bH}_{N-1}\right\} \cup \left\{\bH_{2},\overline{\bH}_{2},\ldots,\bH_{N},\overline{\bH}_{N}\right\}.
\end{equation}
\end{itemize}
The first set on the right in \eqref{eq:partitionfHH} and \eqref{eq:partitionfOE} comprises the \emph{right points}, denoted by $\lambda_k$ , $k=1,\ldots,N$, while the second set comprises the \emph{left points} $\mu_h$ , $h=1,\ldots,N$.

The following step in the Loewner framework is choosing tangential directions as vectors which transform matrix data $\bH_j$ into vector data: \emph{right tangential directions} are column vectors $\br_k \in \CC^{q}$ such that $\bH_k \br_k=\bw_k$, whereas \emph{left tangential directions} are row vectors $\ell_h \in \CC^{1 \times p}$ such that $\ell_h \bH_h =\bv_h$.  The column vectors $\bw_k \in \CC^{p}$ are referred to as \emph{right vector data}, while the row vectors $\bv_h \in \CC^{1 \times q}$ are referred to as \emph{left vector data}. For simplicity, tangential directions can be chosen as alternating columns/rows of the identity matrix \cite{SLACATCAD09}, resulting in vector data being column and row vectors of the original matrix data $\bH_j$ in \eqref{eq:meas}. 

\begin{remark}
For scalar data obtained from single-input single-output (SISO) systems ($p = q = 1$), tangential directions $\br_k$, $\ell_h$ are simply equal to 1.
\end{remark}

\begin{remark}
If the loss of information due to utilizing a single tangential direction per measurement, instead of the whole matrix $\bH_j$, does not allow to obtain an accurate approximation, one can employ the original matrix $\bH_j$. This is equivalent to considering several tangential directions for the same point. To obtain block right matrix data for $\bH_j \in \mathbb{C}^{p \times q}$, the corresponding frequency should be repeated $q$ times as a right point and all columns of the identity matrix of size $q \times q$ should be considered as right directions. Similarly, to obtain block left matrix data for $\bH_j \in \mathbb{C}^{p \times q}$, the corresponding frequency should be repeated $p$ times as a left point and all rows of the identity matrix of size $p \times p$ should be considered as left directions.
\end{remark}

With this notation in place, the Loewner matrix is defined entry-wise as
\begin{equation}\label{eq:Loewner_def}
\LL_{hk} = \frac{\bv_h \br_k-\ell_h \bw_k}{\mu_h-\lambda_k},~~ h,k=1,\ldots,N,
\end{equation} 
and the shifted Loewner matrix is defined as
\begin{equation}\label{eq:shiftedLoewner_def}
\sL_{{hk}} = \frac{\mu_h\bv_h \br_k-\lambda_k\ell_h \bw_k}{\mu_h-\lambda_k},~~ h,k=1,\ldots,N.
\end{equation} 
Note that the numerators are scalar quantities as they are obtained by taking inner products.

The quantities defined previously are collected into the following matrices
\begin{align}\label{eq:right}
\bm \Lambda &=\mbox{diag} \left(\left[ \begin{array}{ccc}
\lambda_1,&\ldots&,\lambda_{N}
\end{array}\right]\right)\in \CC^{N \times N},&
\bR&=\left[\begin{array}{ccc}
\br_1,&\ldots&,\br_{N}
\end{array}\right]\in \CC^{q \times N},& 
\bW&=\left[\begin{array}{ccc}
\bw_1,&\ldots&,\bw_{N}
\end{array}\right]\in \CC^{p \times N},&
\notag\\ 
\bm M &=\mbox{diag}\left(\left[\begin{array}{ccc}
\mu_1,&\ldots&,\mu_{N}
\end{array}\right]\right)\in \CC^{N \times N},& 
\bL&=\left[\begin{array}{c}
\ell_1\\
\vdots\\
\ell_{N}
\end{array}\right]\in \CC^{N \times p},& 
\bV&=\left[\begin{array}{c}
\bv_1\\
\vdots\\
\bv_{N}
\end{array}\right]\in \CC^{N \times q},&
\end{align} \vspace{-5mm}
\begin{align}\label{eq:LLsL}
\LL &= \left[\begin{array}{ccc} 
\frac{\bv_1 \br_1-\ell_1 \bw_1}{\mu_1-\lambda_1}& \ldots &\frac{\bv_1 \br_N-\ell_1 \bw_N}{\mu_1-\lambda_N}\\
\vdots&\ddots&\vdots\\
\frac{\bv_N \br_1-\ell_N \bw_1}{\mu_N-\lambda_1}& \ldots &\frac{\bv_N \br_N-\ell_N \bw_N}{\mu_N-\lambda_N}
\end{array}\right]\in \CC^{N \times N},&
\sL &= \left[\begin{array}{ccc} 
\frac{\bv_1 \br_1-\ell_1 \bw_1}{\mu_1-\lambda_1}& \ldots &\frac{\bv_1 \br_N-\ell_1 \bw_N}{\mu_1-\lambda_N}\\
\vdots&\ddots&\vdots\\
\frac{\bv_N \br_1-\ell_N \bw_1}{\mu_N-\lambda_1}& \ldots &\frac{\bv_N \br_N-\ell_N \bw_N}{\mu_N-\lambda_N}
\end{array}\right]\in \CC^{N \times N}.
\end{align} 

By construction, the Loewner and shifted Loewner matrices satisfy the following Sylvester equations:
\begin{align}\label{eq:Sylv}
\bm M \LL-\LL \bm \Lambda  &= \bV\bR-\bL \bW,&
\bm M \sL-\sL \bm \Lambda  &= \bm M\bV\bR-\bL \bW \bm \Lambda,
\end{align}
as well as the following relations:
\begin{align}\label{eq:relation_LLandSL}
\sL-\LL \bm \Lambda &= \bV\bR,&
\sL- \bm M \LL &= \bL \bW,
\end{align}
which will prove useful in our proposed matrix-free matrix-vector product approach.

Assuming that the data is generated from a real system \eqref{eq:meascc}, to avoid complex arithmetic, a change of basis can be performed. By defining
\begin{equation}\label{eq:changeofbasis}
   \bm \Pi = \frac{1}{\sqrt{2}}\left[\begin{array}{rr} 
1& -\ii \\1&\ii
\end{array}\right] \mbox{ and } \bP = \mbox{blkdiag} \left(\left[ \begin{array}{ccc} 
\bm \Pi,&\ldots&,\bm \Pi
\end{array}\right]\right)\in \CC^{N \times N},
\end{equation}
we obtain matrices with real entries:
\begin{equation*}%
\bm \Lambda_r\hspace{-.5mm}:=\hspace{-.5mm}\bP^*\hspace{-.5mm}\bm \Lambda \bP,~ \bm M_r \hspace{-.5mm}:=\hspace{-.5mm}\bP^*\hspace{-.5mm}\bm M \bP,~ \bL_r\hspace{-.5mm}:=\hspace{-.5mm}\bP^*\bL,~ \bV_r\hspace{-.5mm}:=\hspace{-.5mm}\bP^*\bV,~\bR_r\hspace{-.5mm}:=\hspace{-.5mm}\bR \bP,~\bW_r\hspace{-.5mm}:=\hspace{-.5mm}\bW \bP,~\LL_r\hspace{-.5mm}:=\hspace{-.5mm}\bP^*\LL \bP,~ \sL_r \hspace{-.5mm}:=\hspace{-.5mm}\bP^*\sL \bP,
\end{equation*}
where $\bP^*$ stands for the complex conjugate transpose of the matrix $\bP$ and $\bP^{-1} = \bP^*$.  These quantities satisfy the same equations as in \eqref{eq:Sylv} and \eqref{eq:relation_LLandSL}. 
Unfortunately, $\bm \Lambda_r$ and $\bm M_r$ are no longer diagonal and this represents a major drawback in taking advantage of the Sylvester equations \eqref{eq:Sylv} for a fast computation of $\LL_r$ and $\sL_r$. 
However, $\bm \Lambda_r^2$ and $\bm M_r^2$ are diagonal and given by
\begin{align}\label{eq:matrices_square}
  \bm \Lambda_r^2&=  \mbox{blkdiag} \left[ \begin{array}{rr}
  -\omega_k^2&0\\
  0&-\omega_k^2
  \end{array}\right],\; k=1,\ldots,N,& 
  \bm M_r^2&=  \mbox{blkdiag} \left[ \begin{array}{rr}
  -\omega_h^2&0\\
  0&-\omega_h^2
  \end{array}\right],\; h=1,\ldots,N.
  \end{align}
By multiplying the first equation in \eqref{eq:Sylv} by $\bm M_r$ on the left and, afterwards, multiplying it by $\bm \Lambda_r$ on the right and adding the results together, a new Sylvester equation with diagonal coefficient matrices is obtained:
\begin{align}\label{eq:Sylv_realLL}
\bm M_r^2 \LL_r-\LL_r \bm \Lambda_r^2  &= \bm M_r\left(\bV_r\bR_r-\bL_r \bW_r\right)+\left(\bV_r\bR_r-\bL_r \bW_r\right)\bm \Lambda_r.
\end{align}
By performing the same operations on the second equation in \eqref{eq:Sylv}, a similar Sylvester equation is obtained for the shifted Loewner matrix:
\begin{align}\label{eq:Sylv_realsL}
\bm M_r^2 \sL_r-\sL_r \bm \Lambda_r^2  &= \bm M_r\left(\bm M_r\bV_r\bR_r-\bL_r \bW_r \bm \Lambda_r\right)+ \left(\bm M_r\bV_r\bR_r-\bL_r \bW_r \bm \Lambda_r\right)\bm \Lambda_r.
\end{align}
In the following, we say we employ the {\sc Odd\&Even (real)} partition whenever the approach above is adopted.

After introducing notation, we are ready to state the solution provided by the Loewner framework  to the approximation Problem \ref{approx_pr}. A (non minimal) model for the transfer function in descriptor-form $\bH(\ls) = \bC \left(\ls\bE-\bA\right)^{-1}\bB+\bD$  is given by
\begin{equation}
\bH(s) = \bW \left(\sL-s \LL \right)^{-1} \bV,
\end{equation}
or, alternatively, by $\bH(s) = \bW_r \left(\sL_r-s \LL_r \right)^{-1} \bV_r$ if real arithmetic was enforced. As we have recast the original problem as a tangential interpolation problem, this transfer function satisfies the right and left interpolation conditions \cite{artAJMACA} $\bH(\lambda_k) \br_k\hspace{-1mm}=\hspace{-1mm}\bw_k$ and $\ell_h \bH(\mu_h)\hspace{-1mm}=\hspace{-1mm}\bv_h$, $h,k=1,\ldots,N$ exactly. To obtain a minimal model, we perform a singular value decomposition 
\begin{equation} \label{eq:SVD}
[\bY,\bm \Sigma,\bX] = \mbox{svd}(\sL - x \LL),~~x \in \{f_i\},
\end{equation}
where $\bm \Sigma$ is diagonal and $\bY$,  $\bX$ contain the left and right singular vectors, respectively.
Choosing the order $n$ of the truncated SVD ($n$ is application-dependent), we define (in Matlab notation) $\bX_n = \bX(:,1\hspace{-1mm}:\hspace{-1mm}n)$ and $\bY_n \hspace{-1mm}=\hspace{-1mm} \bY(:,1\hspace{-1mm}:\hspace{-1mm}n)^*$. Finally, the model of size $n$ in descriptor form is  
\begin{equation} \label{eq:E_D0}
\bE = -\bY_n \LL \bX_n=-\LL_n,~
\bA= -\bY_n \sL \bX_n=-\sL_{n},~
\bB = \bY_n \bV=\bV_n,~
\bC = \bW \bX_n=\bW_n,~
\bD=\bfz. 
\end{equation} 

When employing real arithmetic, the SVD trunctation step is analogous, in terms of the $\LL_r$, $\sL_r$, $\bV_r$ and $\bW_r$ matrices. In the following section, we exploit the Cauchy-like structure of the  Loewner and shifted Loewner matrices to design efficient approaches, both in terms of memory storage and CPU time, to compute the SVD in \eqref{eq:SVD} by making use of hierarchical matrices.   

\section{Exploiting the structure of $\mathbb{L}$ and $\mathbb{S}$}\label{Exploiting the structure of L and S}
For data sets with a sizable number $N$ of measurements $\mathbf{H}_j$, the construction of the large, dense Loewner and shifted Loewner matrices 
is demanding, both in terms of computational efforts as well as storage requirements. The computation of each entry of $\LL$ and $\sL$ using \eqref{eq:Loewner_def} and \eqref{eq:shiftedLoewner_def} yields a total cost of $\mathcal{O}(N^2\left(p+q\right))$ floating point operations (FLOPs) for assembling the entire $\LL$ and $\sL$ matrices. The number of nonzero entries in $\LL$ and $\sL$ is $\mathcal{O}(N^2)$, much larger than the memory requirements for storing the data in $\bm{\Lambda}$, $\bm{M}$, $\bR$, $\bW$, $\bL$, and $\bV$\footnote{The number of nonzero entries in the data matrices $\bV$ and $\bR$ amounts to $\mathcal{O}(qN)$ and to $\mathcal{O}(pN)$ for $\bW$ and $\bL$.}. Besides these excessive storage requirements, there are also considerations to be made regarding the CPU time required for the SVD computation of the matrix $\mathbb{S}-x\mathbb{L}$, $x\in\{f_i\}$ in \eqref{eq:SVD}. Especially for large dimensional problems, for which we expect a fast decay, it is preferred to compute only the first $n$ singular triplets, thus avoiding wasting resources in computing the full SVD. To this end, many iterative methods have been developed for computing partial SVDs; see, e.g., \cite{Sto12,Hoc01,BagRei05,Lar98} to name a few. The bottleneck in these approaches is the matrix-vector product with the coefficient matrix, namely $\sL-x\LL$ in our case. This operation costs $\mathcal{O}(N^2)$ FLOPs due to the dense pattern of $\sL-x\LL$.

This section tackles the cost reduction of performing a matrix-vector product with $\sL-x\LL$ while avoiding the explicit allocation of $\mathbb{L}$ and $\mathbb{S}$. The proposed strategy is supported by a thorough analysis of the computational cost, showing that, for very large data sets for which carrying out the full SVD is intractable, our strategy leads to remarkable reductions in both the computational efforts and the storage demand for building minimal realizations in the Loewner framework.


\subsection{Hadamard product and Cauchy matrices}\label{Hadamard product and Cauchy matrices}

We present novel results which exploit the particular structure of the Loewner and shifted Loewner matrices. These developments involve the Sylvester equations \eqref{eq:Sylv} with diagonal coefficient matrices $\bm \Lambda$ and $\bM$. 

\begin{Theorem}\label{Th:expression_LLandsL}
The Loewner and shifted Loewner matrices $\mathbb{L}$ and $\mathbb{S}$ satisfying the Sylvester equations in~\eqref{eq:Sylv} are such that 
\begin{equation}\label{eq:LLCauchy}
\LL=\sum_{j=1}^q \mbox{diag}(\widetilde \bv_j)\C\mbox{diag}(\widetilde \br_j^*)-
\sum_{j=1}^p \mbox{diag}(\widetilde \ell_j)\C\mbox{diag}(\widetilde \bw_j^*),\end{equation}
and
\begin{equation}\label{eq:sLCauchy}
\sL=\sum_{j=1}^q \mbox{diag}(\bm M\widetilde \bv_j)\C\mbox{diag}(\widetilde \br_j^*)-\sum_{j=1}^p \mbox{diag}(\widetilde \ell_j)\C\mbox{diag}( \bm {\Lambda}^*\widetilde \bw_j^*),
\end{equation}
where $\C$ denotes the following Cauchy matrix
\begin{align*}
\C = \left[\begin{array}{ccc} 
\frac{1}{\mu_1-\lambda_1}& \ldots &\frac{1}{\mu_1-\lambda_N}\\
\vdots&\ddots&\vdots\\
\frac{1}{\mu_N-\lambda_1}& \ldots &\frac{1}{\mu_N-\lambda_N}
\end{array}\right],
\end{align*} 
while the vectors $\widetilde \bv_j\in\CC^{N}$ and $\widetilde \ell_j\in\CC^{N}$ denote the $j$-th columns of $\bV$ and $\bL$, respectively, so that
$$\bm V=[\widetilde \bv_1,\ldots, 
\widetilde \bv_q],\quad \bm L=[\widetilde \ell_1,\ldots, 
\widetilde \ell_q].$$
Similarly, the vectors $\widetilde \br_j\in \CC^{1 \times N}$ and $\widetilde \bw_j\in \CC^{1 \times N}$ are the $j$-th rows of $\bR$ and $\bW$, respectively, namely 
$$\bm R=\left[\begin{array}{c}
     \widetilde \br_1  \\
     \vdots \\
     \widetilde \br_q\\
\end{array}\right], \quad 
\bm W=\left[\begin{array}{c}
     \widetilde \bw_1  \\
     \vdots \\
     \widetilde \bw_q\\
\end{array}\right].
$$
\end{Theorem}
\begin{proof}
The Loewner and shifted Loewner matrices $\LL$ and $\sL$
are Cauchy-like matrices as they are obtained by taking the Hadamard product $\circ$ between the Cauchy matrix $\C$ 
and the right-hand sides of the Sylvester equations in \eqref{eq:Sylv}. In particular,
\begin{align}
\LL &= \C \circ\left(\bV\bR-\bL \bW\right),&
\sL &= \C \circ\left(\bm M\bV\bR-\bL \bW \bm \Lambda\right).
\end{align}
An important property of the Hadamard product reads as follows. For any vectors $\bx,\by\in\mathbb{C}^{N}$, it holds
$$\C\circ(\bx\by^*)=\text{diag}(\bx)\C\text{diag}(\by^*). $$
This, along with the low-rank structure of $\bV\bR-\bL \bW $ and $\bm M\bV\bR-\bL \bW \bm \Lambda$, yields the results in \eqref{eq:LLCauchy} and~\eqref{eq:sLCauchy}.
\end{proof}
\begin{Cor}\label{Cor_matrixvecprod}
Given a vector $\by \in\mathbb{C}^{N}$ and $x\in\{f_i\}$, we have 
$$(\sL-x\LL)\by=\sum_{j=1}^q \widetilde \bv_j\circ(\C(\widetilde \br_j^*\circ\widetilde \by))-\sum_{j=1}^p\widetilde \ell_j\circ(\C(\widetilde \bw_j^*\circ\widetilde \by))+\bV \bR \by,$$
where $\widetilde \by=(\bm \Lambda-x\bI)\by$, with $\bI$, the identity matrix. 
\end{Cor}
\begin{proof}
Thanks to \eqref{eq:relation_LLandSL}, we can write
\begin{align*}
(\sL-x\LL)\by=& (\LL\bm\Lambda + \bV\bR -x\LL)\by=\LL(\bm\Lambda-x\bI)\by+  \bV\bR \by.
\end{align*}
The result follows by substituting the expression of $\LL$ given in Theorem~\ref{Th:expression_LLandsL} in the equation above.

\end{proof}

Similar results to those in Theorem~\ref{Th:expression_LLandsL} and Corollary \ref{Cor_matrixvecprod} can be obtained for $\LL_r$ and $\sL_r$ solving the Sylvester equations in \eqref{eq:Sylv_realLL}
and \eqref{eq:Sylv_realsL}, respectively. The developments follow the same line of proof as above with straightforward adjustments.

Corollary \ref{Cor_matrixvecprod} shows that the majority of the computational cost of performing the matrix-vector multiplication $(\sL-x\LL)\by$ amounts to computing $p+q$ matrix-vector products with the Cauchy matrix $\C$.

Extensive research has been devoted to fully exploiting the rich structure of Cauchy matrices. Several algorithms for computing the matrix-vector product $\C \by$ can be found in the literature and many avoid assembling the full matrix $\C$ (see, e.g., \cite{Pan2014,GREENGARD1987325,Gohberg1994,Carrier1988}). In the next section we 
recall the strategy presented by Pan in \cite{PAN2015107,Pan2014} to represent $\C$ in terms of a hierarchically semiseparable (HSS) matrix. Even though the novel scheme proposed in this paper does not depend on the strategy employed for performing the matrix-vector product $\C \by$ -- as long as it is efficient -- we believe that the HSS framework may be advantageous as, in principle, many matrix-vector products with $\C$ are needed for computing a (partial) SVD of the matrix $\sL-x\LL$.

We conclude this section with the following remarks.

\begin{remark}\label{remark:numericalrankC}
The number $n$ of singular triplets needed to be computed to achieve the minimal realization $(\bE,\bA,\bB,\bC,\bD)$ in \eqref{eq:E_D0} is difficult to estimate a-priori\footnote{In \cite[Section 4.3]{Beckermann2019}, some results on the numerical rank of $\LL$ are presented provided \mbox{$[\min\lambda_k,\max\lambda_k]\cap [\min\mu_h,\max\mu_h]=\emptyset$.}}. However, the expression of $\LL$ and $\sL$ in terms of the Hadamard product can be useful to this end. Indeed, another important property of the Hadamard product is that, for any matrices $\bA$ and $\bB$, $\text{rank}(\bA\circ \bB)\leq\text{rank}(\bA)\text{rank}(\bB)$. Therefore,
\begin{align*}
\text{rank}(\LL) &\leq\text{rank}(\C)\cdot\text{rank}(\bV\bR-\bL \bW)\leq (p+q)\cdot\text{rank}(\C),
\end{align*}
and similarly for $\sL$. Thus, we have 
\begin{align}\label{eq:bound_rank_L}
\text{rank}(\sL-x\LL) &\leq2(p+q)\cdot\text{rank}(\C), \quad \forall\,x\in\{f_i\}.
\end{align}

In general, the Cauchy matrix $\C$ is full rank so this inequality is trivially satisfied. However, depending on the partitioning of the points into $\lambda_k$ and $\mu_h$ (as in \eqref{eq:partitionfHH} and \eqref{eq:partitionfOE}), it can be \emph{numerically} low-rank (see, e.g., \cite[Theorem 5]{Pan2014}, \cite{Beckermann2019,Chandrasekaran2007}).
If $\pi_\C$ denotes the numerical rank of $\C$, then $2(p+q)\pi_\C$ is a rough estimate for the numerical rank of $\sL-x\LL$\footnote{For $\LL_r$ and $\sL_r$ satisfying \eqref{eq:Sylv_realLL} and \eqref{eq:Sylv_realsL}, the value $4(p+q)\pi_\C$ can be used as an estimate for the number of singular triplets to compute.}. Oftentimes, the underlying dynamical system is of much lower complexity, thus allowing for the computation of a minimal realization of reduced order $n$. One can also use insight of the system itself or count the number of peaks in the frequency response to estimate $n$ (for systems with poles having dominant imaginary parts). 
\end{remark}

\begin{remark}
The expression of $\LL$ and $\sL$ in terms of the Hadamard product provides us with an upper bound of the spectral norm of the Loewner and shifted Loewner matrix. 
Indeed, the spectral norm is submultiplicative with respect to the Hadamard product \cite[Theorem 5.5.1]{Horn1991}, hence 
$$
\begin{array}{rll}
    \|\LL\|&=&\|\C\circ(\bV\bR-\bL \bW)\|\leq\|\C\|\cdot\|\bV\bR-\bL \bW\|
    \\
   & \leq&
    \|\C\|_F\cdot\|\bV\bR-\bL \bW\|
    =
    \|\bV\bR-\bL \bW\|\sqrt{\sum_{i=1}^N\sum_{j=1}^N\left|\frac{1}{\mu_i-\lambda_j}\right|^2},\\
\end{array}
$$
where $\|\C\|_F$ denotes the Frobenius norm of $\C$. Note that $\|\bV\bR-\bL \bW\|$ can be computed cheaply, e.g., by a power method exploiting the low rank of $\bV\bR-\bL \bW$.

Similarly, 
$$\|\sL\|\leq \|\bM\bV\bR-\bL \bW\bm \Lambda\|\sqrt{\sum_{i=1}^N\sum_{j=1}^N\left|\frac{1}{\mu_i-\lambda_j}\right|^2}.$$
\end{remark}

\begin{remark}
Low-rank approximations to $\LL$ and $\sL$ may be computed by adaptive cross approximation \cite{ACA}, particularly suited for hierarchical matrices, the CUR decomposition~\cite{CUR} as in \cite{LoewCUR, 9073015}, or related schemes. These approaches select a certain number of columns and rows of the original matrices in a greedy fashion based on various heuristics, and a \emph{core} matrix is utilised to compute a low-rank approximation. If a given threshold on the desired accuracy of the computed approximation is provided as an input, these algorithms often construct matrices whose rank is much larger than the one of the target matrices $\LL$ and $\sL$. On the other hand, by fixing the rank $k$ of the approximation, $k\approx \mathtt{rank}(\LL),\mathtt{rank}(\sL)$ - assuming we know an estimate of $\mathtt{rank}(\LL)$, $\mathtt{rank}(\sL)$ - the accuracy we achieve may be very low affecting the reliability of the computed reduced models.
\end{remark}
\subsection{Hierarchically semiseparable (HSS) representation of a Cauchy matrix}\label{HSS representation of a Cauchy matrix}

The literature on HSS matrices is rather vast and technical (see, e.g., \cite{Vandebril2005,Pan2014,PAN2015107,Chandrasekaran2006,Xia2010} and references therein). Here we recall only the main properties of this class of matrices and their role in the efficient representation of Cauchy matrices. Such a technique is also closely related to the Fast Multipole Method (FMM). We refer the interested reader to, e.g.,  \cite{Chandrasekaran2007,Chandrasekaran2006} for more details on the interconnection between HSS matrices and FMM.

\begin{Definition}[{\protect{\cite[Definition 27]{PAN2015107}}}]\label{Def:hss}
Let $\bA$ be an $N\times N$ matrix with $\alpha$ being the maximum rank of all its subdiagonal blocks, namely the blocks of all sizes lying strictly below the block diagonal, and $\beta$ the  maximum rank of all its superdiagonal blocks, namely the blocks of all sizes lying strictly above the block diagonal, respectively.
Then, $\bA$ is $(\alpha,\beta)$-HSS if its diagonal blocks consist of $\mathcal{O}((\alpha+\beta)N)$ entries. 
\end{Definition}

The $(\alpha,\beta)$-HSS representation of a matrix $\bA$ is very advantageous whenever $\alpha$ and $\beta$ are small. 
For instance, it allows us to express $\bA$ in terms of $\mathcal{O}((\alpha+\beta)N)$ parameters avoiding storing its $N^2$ entries. Moreover, a whole, efficient HSS arithmetic has been developed in the last decades (see, e.g., \cite{Chandrasekaran2006,Xia2010}). For instance, the computational cost of the matrix-vector product $\bA\by$ amounts to $\mathcal{O}((\alpha+\beta)N)$ FLOPs. 
If $\bA$ is nonsingular, its inverse is also a $(\alpha,\beta)$-HSS matrix that can be computed in $\mathcal{O}((\alpha+\beta)^3N)$ FLOPs (see, e.g., \cite[Section 6]{PAN2015107}). 

To fully exploit the HSS framework for our purposes, we wish to represent the Cauchy matrix $\C$ in terms of a HSS matrix with a low off-diagonal rank. In light of Corollary~\ref{Cor_matrixvecprod}, this would considerably decrease the computational cost of the matrix-vector products involving $\sL-x\LL$ while avoiding forming the dense matrices $\sL$ and $\LL$.  

The construction of an HSS approximation $\widetilde\C$ to $\C$ is rather involved and the magnitude of the $(\alpha,\beta)$-rank of the computed $\widetilde\C$ strictly depends on the partitioning of the frequencies along with the accuracy that has been selected for the actual computation of $\widetilde\C$\footnote{Roughly speaking, such a threshold is related to the computation of the low-rank approximations to the off-diagonal blocks of $\C$ (see, e.g., \cite[Corollary 4.3]{Xi2014}, \cite[Theorem 4.7]{Kressner2019}).} (see, e.g., \cite[Section 8]{PAN2015107} for further details on the computation of an HSS-representation of a Cauchy matrix). In this paper we employ the readily available {\tt hm-toolbox} \cite{hmtoolbox}.


\begin{num_example}\label{Ex.1}
{\rm
We investigate the impact of the most commonly-used frequency partitions ({\sc Half\&Half}, {\sc Odd\&Even}, {\sc Odd\&Even (Real)}) on the HSS-rank of the computed $\widetilde\C$ for a mechanical structure. We emphasize that the most effective partition is problem-dependent and is still an open problem, beyond the scope of this paper. We consider the \emph{Flexible Aircraft} data set \cite{POUSSOTVASSAL2018559} from the {\sc MORwiki}~\cite{morWiki}.
This dataset contains 421 frequency values $\omega_j$ expressed in rad/s and the corresponding measurements of the transfer function $\bH_j$. We disregard the last data point and consider the remaining frequencies ranging from $f_1 = 0.1$Hz to $f_{420}=42$Hz. As this is a mechanical structure, frequencies considered are in the low spectrum, as opposed to electrical systems, for which frequencies typically span the GHz range.

We recall the three different partitions of the frequencies $\{\omega_j\}_{j=1}^{j=420}$: 
\begin{itemize}
    \item {\sc Half\&Half}: $\bm \Lambda=\text{diag}([\ii\omega_1,-\ii\omega_1,\ldots,\ii\omega_{210},-\ii\omega_{210}])$,  
    $\bM=\text{diag}([\ii\omega_{211},-\ii\omega_{211},\ldots,\ii\omega_{420},-\ii\omega_{420}])$.
     \item {\sc Odd\&Even}: $\bm \Lambda=\text{diag}([\ii\omega_1,-\ii\omega_1,\ldots,\ii\omega_{419},-\ii\omega_{419}])$, 
     $\bM=\text{diag}([\ii\omega_{2},-\ii\omega_{2},\ldots,\ii\omega_{420},-\ii\omega_{420}])$.
     \item {\sc Odd\&Even \hspace{-1mm}(\hspace{-.5mm}Real\hspace{-.5mm})}: $\bm \Lambda_r\hspace{-.5mm}=\hspace{-.5mm}\text{diag}([-\omega_1^2,\hspace{-.5mm}-\omega_1^2,\ldots,\hspace{-.5mm}-\omega_{419}^2,\hspace{-.5mm}-\omega_{419}^2])$, 
     $\bM_r\hspace{-.5mm}=\hspace{-.5mm}\text{diag}([-\omega_{2}^2,\hspace{-.5mm}-\omega_{2}^2,\ldots,\hspace{-.5mm}-\omega_{420}^2,\hspace{-.5mm}-\omega_{420}^2])$.
\end{itemize}

\begin{table}[htbp!]
\arraycolsep=0.5pt\def\arraystretch{1.5}
 \centering
  \begin{tabular}{rrrr}
    & {\sc Half\&Half} & {\sc Odd\&Even} & {\sc Odd\&Even (Real)} \\
    \hline
    $\mathtt{hssrank}(\widetilde\C)$ & 32 & 30 & 13 \\
    $\mathtt{rank}(\C)$ & 36 & 420 & 210 \\
    $\|\widetilde\C-\C\|/\|\C\|$& 2.68e-12 & 2.61e-11 & 6.62e-13
    \\ 
      \end{tabular}
\caption{Example \ref{Ex.1}. HSS-rank and relative error of the HSS representation $\widetilde\C$ of $\C$ for different frequency partitions along with the rank of $\C$.}\label{tab1.ex1}
 \end{table}
For each partition, we compute the corresponding Cauchy matrix in HSS format $\widetilde\C$ without assembling the full $\C$ beforehand, by means of the function {\tt hss} of the {\tt hm-toolbox}:
$$\widetilde \C =\mathtt{hss('cauchy',dM,-dL,N,N)}$$
where $\mathtt{dM}$ and $\mathtt{dL}$ are $N$ dimensional vectors containing the frequencies $\mu_h$ and $\lambda_k$, respectively. We then calculate its rank by $\mathtt{hssrank}(\widetilde\C)$\footnote{Following Definition~\ref{Def:hss}, this function returns $\max\{\alpha,\beta\}$.}. In Table~\ref{tab1.ex1} we report the HSS-rank of the matrix $\widetilde\C$ for  the partitions mentioned above. Thanks to the small dimension of the dataset, we are able to compute the full Cauchy matrix $\C$ and document its (standard) rank along with the relative error $\|\widetilde\C-\C\|/\|\C\|$. As expected, having two disjoint sets of frequencies like in the {\sc Half\&Half} partition leads to a Cauchy matrix $\C$ whose (standard) rank is low. This does not happen in the other two scenarios we examine so that taking advantage of the HSS format is necessary to achieve memory-saving representations of $\C$. The results in Table~\ref{tab1.ex1} show that a good accuracy in terms of the relative error can be achieved for all three frequency partitions. Nevertheless, the HSS rank of $\widetilde\C$ is significantly lower for the {\sc Odd\&Even (Real)} partition, most likely due to the squaring of the frequencies performed in {\sc Odd\&Even (Real)}, which leads to a fast decay in the magnitude of the off-diagonal entries of $\C$. Hence, for a fixed threshold, the off-diagonal blocks of the Cauchy matrix associated to the {\sc Odd\&Even (Real)} partition
 can be approximated by matrices having a smaller rank than those associated to the other two scenarios we examined.

 In Figure~\ref{fig1.ex1} we display the absolute value -- on a logarithmic scale -- of the entries of the Cauchy matrix $\C$ stemming from the different partitions. The same scale has been used in all the three figures, enforcing the observation that the {\sc Odd\&Even (Real)} partition exhibits the fastest decay in the magnitude of the off-diagonal entries of $\C$ .
 
  \begin{figure}
\begin{minipage}[c]{.28\textwidth}
\centering
\caption*{{\sc Half\&Half}}
\includegraphics[scale=0.4]{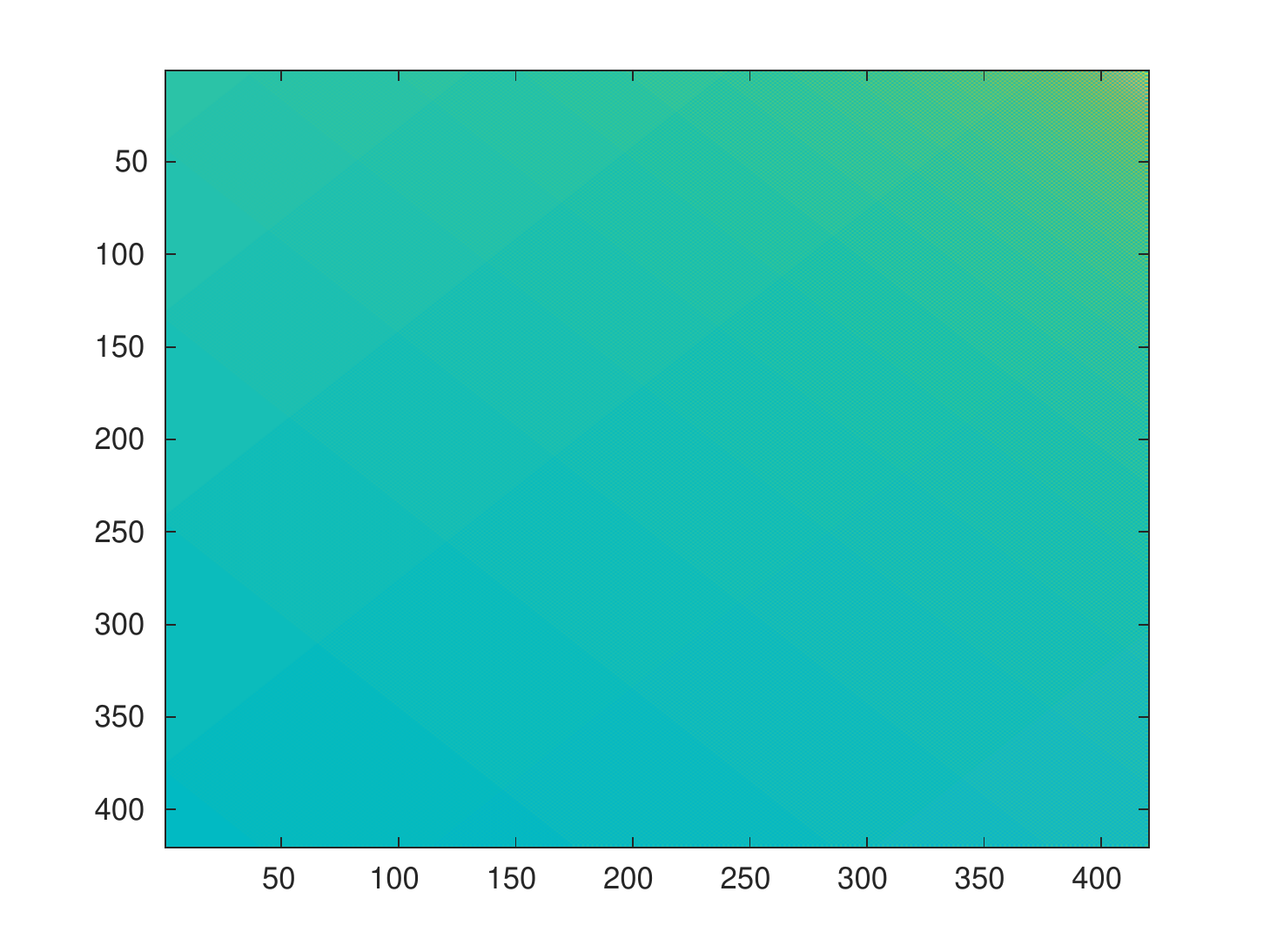}%
\end{minipage}%
\hspace{10mm}%
\begin{minipage}[c]{.28\textwidth}
\centering
\caption*{{\sc Odd\&Even}}
\includegraphics[scale=0.4]{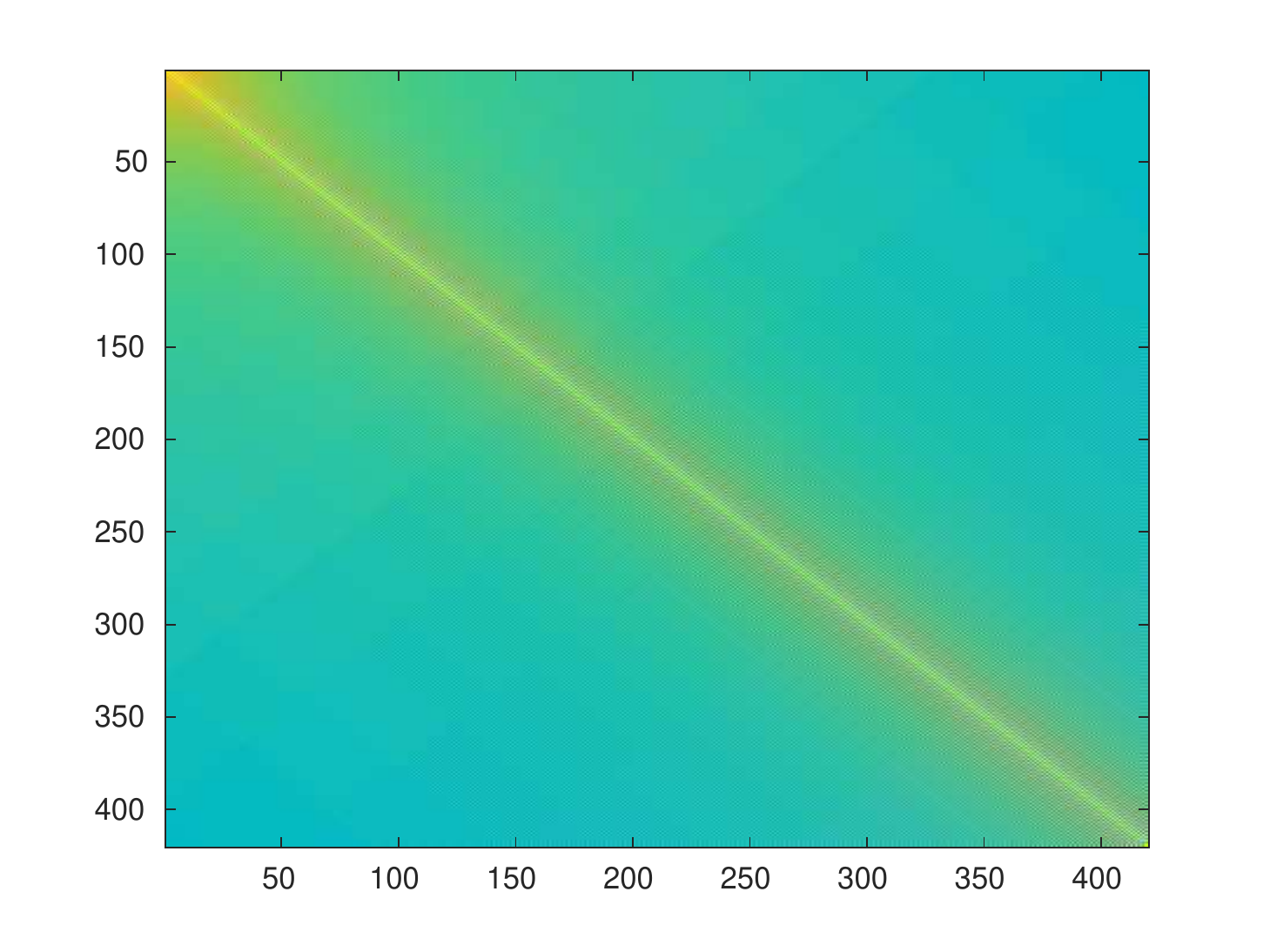}%
\end{minipage}
\hspace{10mm}%
\begin{minipage}[c]{.28\textwidth}
\centering
\caption*{{\sc Odd\&Even (real)}}

\includegraphics[scale=0.4]{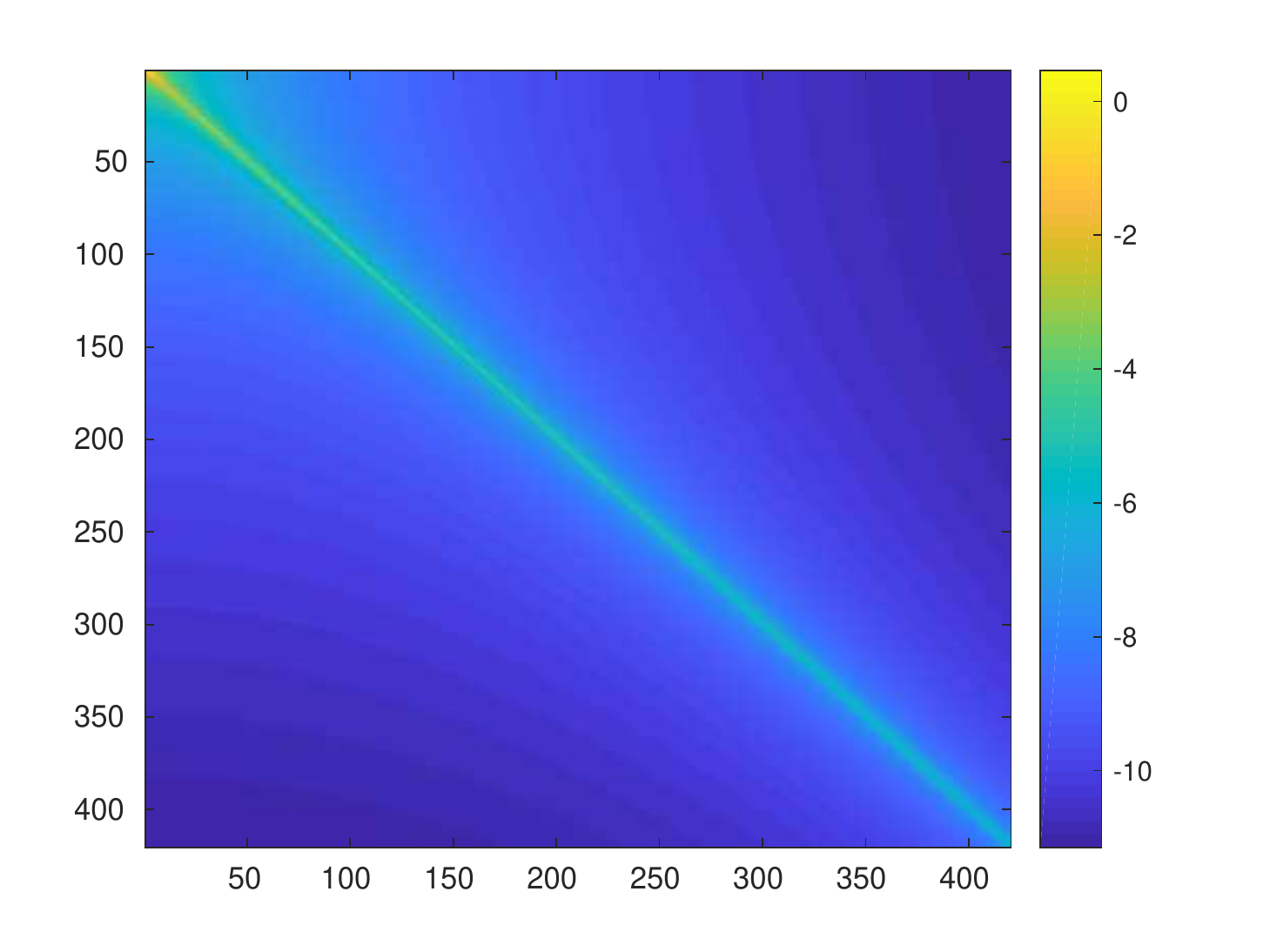}%
\end{minipage}
\caption{Example~\ref{Ex.1}. Absolute value -- on a logarithmic scale -- of the entries of the Cauchy matrix $\C$ stemming from the three different frequency partitions we have examined.}\label{fig1.ex1}
\end{figure}

}
\end{num_example}

\subsection{Efficient, inexact matrix-vector products}
Whenever the matrix $\C$ admits an accurate approximation in terms of a low-rank HSS matrix $\widetilde \C$, the computational cost of performing the matrix-vector product $(\sL-x\LL)\by$ can be significantly reduced.
\begin{Prop}\label{Prop:inexact_matrixvec}
Let $\widetilde \C$ be an $(\alpha,\beta)$-HSS matrix that approximates the Cauchy matrix $\C$ accurately. If $\LL$ and $\sL$ satisfy the Sylvester equations in \eqref{eq:Sylv}, then 
\begin{equation}\label{eq:P5eq1}
(\sL-x\LL)\by=\sum_{j=1}^q \widetilde \bv_j\circ(\widetilde\C(\widetilde \br_j^*\circ\widetilde \by))-\sum_{j=1}^p\widetilde \ell_j\circ(\widetilde\C(\widetilde \bw_j^*\circ\widetilde \by))+\bV \bR \by +  \mathcal{\bm E}\widetilde \by,
\end{equation}
where $\|\mathcal{\bm E}\|\leq(p+q)\|\C-\widetilde\C\|\max^2\{\|\bV\|_{\infty},\|\bR\|_{\infty},\|\bL\|_{\infty},\|\bW\|_{\infty}\}$.
Moreover, the computational cost of performing
\begin{equation}\label{eq:approx}
\sum_{j=1}^q \widetilde \bv_j\circ(\widetilde\C(\widetilde \br_j^*\circ\widetilde \by))-\sum_{j=1}^p\widetilde \ell_j\circ(\widetilde\C(\widetilde \bw_j^*\circ\widetilde \by))+\bV \bR \by,
\end{equation}
amounts to $\mathcal{O}((p+q)(\alpha+\beta+1)N)$ FLOPs.
\end{Prop}
\begin{proof}
From the result in Corollary~\ref{Cor_matrixvecprod}, we can write
\begin{align*}
(\sL-x\LL)\by=&\sum_{j=1}^q \widetilde \bv_j\circ(\widetilde\C(\widetilde \br_j^*\circ\widetilde \by))-\sum_{j=1}^p\widetilde \ell_j\circ(\widetilde\C(\widetilde \bw_j^*\circ\widetilde \by))+\bV\bR \by \\
&+\sum_{j=1}^q \widetilde \bv_j\circ((\C-\widetilde\C)(\widetilde \br_j^*\circ\widetilde \by))-\sum_{j=1}^p\widetilde \ell_j\circ((\C-\widetilde\C)(\widetilde \bw_j^*\circ\widetilde \by))\\
=& \sum_{j=1}^q \widetilde \bv_j\circ(\widetilde\C(\widetilde \br_j^*\circ\widetilde \by))-\sum_{j=1}^p\widetilde \ell_j\circ(\widetilde\C(\widetilde \bw_j^*\circ\widetilde \by))+\bV\bR \by+\mathcal{\bm E}\widetilde \by,
\end{align*}
where $\mathcal{\bm E}:=\displaystyle\sum_{j=1}^q \text{diag}(\widetilde \bv_j)(\C-\widetilde\C)\text{diag}(\widetilde \br_j^*)-\sum_{j=1}^p\text{diag}(\widetilde \ell_j)(\C-\widetilde\C)\text{diag}(\widetilde \bw_j^*)$. Therefore, 
\begin{align*}
    \|\mathcal{\bm E}\|\leq&\|\C-\widetilde\C\|\left(\sum_{j=1}^q\|\text{diag}(\widetilde \bv_j)\|\|\text{diag}(\widetilde \br_j^*)\|+\sum_{j=1}^p\|\text{diag}(\widetilde \ell_j)\|\|\text{diag}(\widetilde \bw_j^*)\|\right)\\
    =& \|\C-\widetilde\C\|\left(\sum_{j=1}^q\|\widetilde \bv_j\|_{\infty}\|\widetilde \br_j\|_{\infty}+\sum_{j=1}^p\|\widetilde \ell_j\|_{\infty}\|\widetilde \bw_j\|_{\infty}\right)\\
    \leq & (p+q)\|\C-\widetilde\C\|\max{^2}\{\|\bV\|_{\infty},\|\bR\|_{\infty},\|\bL\|_{\infty},\|\bW\|_{\infty}\}.
\end{align*}
This proves the first part of Proposition \ref{Prop:inexact_matrixvec}. To conclude, by making use of the property that the matrix-vector product with a $(\alpha,\beta)$-HSS matrix costs $\mathcal{O}((\alpha+\beta)N)$ FLOPs and that $\bV\bR$ has rank $q$, a direct computation shows that the number of operations needed to perform \eqref{eq:approx} amounts to $\mathcal{O}((p+q)(\alpha+\beta+1)N)$ FLOPs, which proves the second claim in Proposition \ref{Prop:inexact_matrixvec}.

\end{proof}
As before, analogous results can be obtained for $\LL_r$ and $\sL_r$ satisfying \eqref{eq:Sylv_realLL}
and \eqref{eq:Sylv_realsL}, respectively.

Proposition~\ref{Prop:inexact_matrixvec} shows that, whenever $\|\C-\widetilde\C\|$ is small, the matrix-vector product $(\sL-x\LL)\by$ can be well-approximated by the expression in \eqref{eq:approx} while dramatically reducing the computational complexity from $\mathcal{O}(N^2)$ FLOPs to $\mathcal{O}((p+q)(\alpha+\beta+1)N)$ FLOPs. However, when this approximation is used within our favorite iterative procedure for computing a partial SVD of $\sL-x\LL$, the inexactness introduced by neglecting the term $\mathcal{\bm E}\widetilde \by$ should be taken into account. 

The use of inexact matrix-vector products within certain iterative procedures has been the subject of numerous research papers: Krylov techniques for solving linear systems and matrix equations \cite{Simoncini2003,Bouras2005,Eshof2004,Kuerschner2019,Kuerschner2018}, eigenvalue problems \cite{Freitag2007,Simoncini2002}, or an inexact variant of the Lanczos bidiagonalization for the computation of some leading singular triplets of a generic matrix function $f(\bA)$ can be found in \cite{GAAF2017}.
With the goal to decrease the computational cost of the overall procedure, these studies show that the accuracy of the matrix-vector product can be \emph{relaxed} (becoming more and more inaccurate) as iterations proceed. In our framework, the inexactness introduced by approximating
$(\sL-x\LL)\by$ with \eqref{eq:approx}
is fixed throughout the entire iterative procedure and mainly depends on $\|\C-\widetilde\C\|$, which is often small, as shown in Example~\ref{Ex.1}. Therefore, the approximation
$$(\sL-x\LL)\by\approx\sum_{i=1}^q \widetilde \bv_i\circ(\widetilde\C(\widetilde \br_i^*\circ\widetilde \by))-\sum_{i=1}^p \widetilde \ell_i\circ(\widetilde\C(\widetilde \bw_i^*\circ\widetilde \by))+\bV \bR \by,$$
does not greatly affect the accuracy of the computed singular triplets (see section~\ref{Numerical results}). Moreover, in our case, we do not need an accurate approximation of the singular triplets of $\sL-x\LL$.
The main goal is to have meaningful spaces spanned by the computed left and right singular vectors so that the obtained reduced model $(\bE,\bA,\bB,\bC)$ inherits the desired approximation properties. 
Moreover, as shown in \cite[Corollary~1.4]{IonitaPhD}, \cite[Proposition 8.25]{TutorialLoewner}, in the case of noise-free measurements of a low-order rational function, even general projectors, not necessarily obtained from the SVD, can be employed for identifying the underlying function. 

\begin{remark}
In Remark~\ref{remark:numericalrankC} we suggested to use the value $2(p+q)\pi_C$, where $\pi_C$ is the numerical rank of $\C$, to decide on the number $n$ of singular triplets of $\sL-x\LL$ needed for the reduced model. For interlaced partitions, as it is the case with  {\sc Odd\&Even} and {\sc Odd\&Even (real)} (see Table \ref{tab1.ex1}), the numerical (standard) rank of the Cauchy matrix is large, in general. Hence, the value 
$n=2(p+q)\max\{\alpha,\beta\}$ may instead be employed for the computation of a meaningful reduced model whenever $\C$ can be well-approximated by a $(\alpha,\beta)$-HSS matrix $\widetilde\C$\footnote{As before, the value $4(p+q)\max\{\alpha,\beta\}$ should be preferred whenever $\LL_r$ and $\sL_r$ solve \eqref{eq:Sylv_realLL} and \eqref{eq:Sylv_realsL}, respectively.}. Moreover, the HSS-rank of 
$\widetilde\C$ is obtained as a byproduct of the construction of $\widetilde\C$.
\end{remark}

\begin{remark} \label{remark:newn}
If $\C$ admits an accurate approximation in terms of an $(\alpha,\beta)$-HSS matrix $\widetilde\C$, the expression in Theorem~\ref{Th:expression_LLandsL} shows that $\LL$ can also be well-approximated by a HSS matrix $\widetilde\LL$ whose rank is at most $((p+q)\alpha,(p+q)\beta)$. Even though the computational cost of $\widetilde\LL \by$ would still be $\mathcal{O}((p+q)(\alpha+\beta)N)$ FLOPs, using the HSS approximation $\widetilde\LL$ of $\LL$ may be very advantageous whenever linear systems with $\LL$ need to be solved (see, e.g., the procedure presented in \cite{Embree2019} for the pseudospectra computation of $\sL-x\LL$). Indeed, as mentioned in section~\ref{HSS representation of a Cauchy matrix}, the computation of the inverse $\widetilde\LL^{-1}$ of $\widetilde\LL$ costs $\mathcal{O}((p+q)^3(\alpha+\beta)^3N)$ FLOPs. Once $\widetilde\LL^{-1}$ is computed, we need only $\mathcal{O}((p+q)(\alpha+\beta)N)$ FLOPs to perform $\widetilde\LL^{-1}\by$.
\end{remark}


\section{Numerical results}\label{Numerical results}

In this section we present numerical experiments illustrating the potential of the proposed approach. 

In Example~\ref{Ex.2}, we compare our approach to standard procedures employed in the Loewner framework. Recall that the main steps in the standard approach involve forming the full Loewner and shifted Loewner matrices $\LL$ and $\sL$ and computing the SVD of $\sL-x\LL$. This SVD can be either computed in full, followed by keeping only the $n$ dominant singular vectors, or only these $n$ singular vectors can be obtained by means of an iterative procedure, where the matrix-vector product with $\sL-x\LL$ is needed\footnote{We employ the \MATLAB functions {\tt svd} and {\tt svds}, respectively.}. In the following, we report the overall running time, considering the \emph{construction} step ({\sc Construction}), i.e., the computation of $\LL$ and $\sL$ in the standard approach and of $\widetilde \C$ in our approach, as well as the \emph{reduction} step ({\sc Reduction}), involving the SVD computation followed by projection to obtain the reduced matrices in~\eqref{eq:E_D0}. In terms of memory requirements, for our approach, this involves the allocation of 
$\widetilde \C$ in the HSS format, while for the standard approach, we report the storage required for $\LL$ and $\sL$. 

In Table~\ref{tab.cost} we recall the computational cost of the construction and reduction steps of both the standard approach, based on either a full or a partial SVD, and the novel one presented in this paper along with their memory requirements.

\begin{table}[h]
\arraycolsep=0.5pt\def\arraystretch{1.5}
 \centering
  \begin{tabular}{r|rrr}
   &  {\sc Construction} & {\sc Reduction} & {\sc Storage}\\
 \hline
  Full {\tt svd} & $\mathcal{O}(N^2(p+q))$ & $\mathcal{O}(N^3)$ &  $\mathcal{O}(N^2)$ \\
  {\tt svds} w/ $\sL-x\LL$ & $\mathcal{O}(N^2(p+q))$ & $\mathcal{O}(nN^2)$ & $\mathcal{O}(N^2)$ \\
  {\tt svds} w/ $\widetilde\C$ & $\mathcal{O}(\max\{\alpha,\beta\}N\log N)$& 
 $\mathcal{O}(n(p+q)(\alpha+\beta+1)N)$ & $\mathcal{O}((\alpha+\beta+p+q)N)$\\
      \end{tabular}
\caption{Computational cost of the construction ({\sc Construction}) and reduction ({\sc Reduction}) steps of the different approaches we test along with their storage demand ({\sc Storage}). The computational cost of the construction of $\widetilde\C$ can be found, e.g., in \cite[Table 1]{hmtoolbox}.}\label{tab.cost}
 \end{table}
 
Lastly, the accuracy of the reduced models is reported in terms of the normalized $\mathcal{H}_2$-error: 
$$\mathcal{H}_2-\text{error}=\sqrt{\frac{\sum_{j=1}^N\|\bH_j-\bH(i\omega_j)\|_F^2}{\sum_{j=1}^N\| \bH_j\|_F^2}},$$
where $\|\cdot\|_F$ denotes the Frobenius norm. Similar results in terms of accuracy are attained for the $\mathcal{H}_\infty$-error, however, we decided not to document them here, for the sake of brevity. 

In Example~\ref{Ex.3}, we compare our novel strategy to the one presented in \cite{FSVDL}, which makes use of the low-rank ADI-Galerkin method for computing the Loewner matrix as the solution to~\eqref{eq:Sylv}. Such a scheme computes low-rank approximations to the dense Loewner matrix to speed-up the SVD computation, however, the memory constraints originating from the allocation of $\LL$ and $\sL$ are still present.

Results were obtained by running \MATLAB R2020b \cite{MATLAB2020b} on a 
MacBook Pro with an Intel Core i9 processor running at 2.3GHz using 16GB of RAM.
All computations involving HSS matrices employed the {\tt hm-toolbox} \cite{hmtoolbox} with the default settings and the threshold for off-diagonal truncation set to $10^{-14}$.

\begin{num_example}\label{Ex.2}
{\rm
We consider a synthetic problem for which we can control the order of the original system~($n$), the number of inputs and outputs ($p=q$), as well as the number of measurements ($N$). The system dynamics is generated randomly, with poles in complex conjugate pairs. In particular:
\begin{itemize}
\item the real part of the poles is random with mean $-10^4$ and standard deviation $-2\cdot 10^3$;  
the imaginary part is also random, with mean $10^4$ and standard deviation $10^6$.
\item residues associated to each pole are rank-1 matrices, obtained as outer products between two random vectors, both having the real part with mean $0$ and standard deviation $10$, while the imaginary part has mean $0$ and standard deviation $10^2$.
\end{itemize}
Measurement points $\{\omega_j\}_{j=1}^{j=N}$ are logarithmically distributed between $10^4$ and $10^7$ rad/sec. Last, but not least, random noise with a signal-to-noise ratio $SNR=100$ was added to the transfer function evaluation $\bH(\ii \omega_j)$ to obtain the measurement matrices $\bH_j$.
We adopt the {\sc Odd\&Even (real)} partition of the frequencies as it achieves satisfactory approximation results while eliminating complex arithmetic.

We compare the proposed approach to the traditional Loewner framework, in which the Loewner and shifted Loewner matrices $\LL$ and $\sL$ are formed and the full SVD of $\sL-x\LL$ is computed, as well as the alternative approach in which, after building $\LL$ and $\sL$, a partial SVD of $\sL-x\LL$ using the Matlab {\tt svds} function is computed for various instances of the data set described above for different values of $N$, $p$, and $n$. The command {\tt svds} was employed with the left starting vector $\widetilde \bv_1$ (same notation as in Theorem~\ref{Th:expression_LLandsL}) instead of a random starting vector, which is the default setting. 

Figure~\ref{fig1.ex2.laptop} presents the memory requirements for storing the Loewner and shifted Loewner matrices $\LL$ and $\sL$ (in red), as opposed to storing the HSS approximation $\widetilde \C$ in our approach (in blue), along with the storage needed to allocate the data in $\bm\Lambda_r$, $\bm M_r$, $\bR_r$, $\bL_r$, $\bV_r$, $\bW_r$, for increasing values in the number of inputs and outputs $p$ (in black). We point out that for values of $N$ larger than $40\,000$, we were not able to allocate the full matrices $\LL$ and $\sL$ on the employed laptop (this value, however, depends on the available RAM memory of the machine). For instances when these matrices can be allocated, Figure~\ref{fig1.ex2.laptop} shows that the memory requirements for the proposed approach are always much lower than for the standard scheme. Moreover, in contrast to what happens to the memory required for the data matrices, the storage demanded by the allocation of $\widetilde \C$ in HSS format is independent of $p$.

 \begin{figure}[t]
\centering
\includegraphics[scale=0.42]{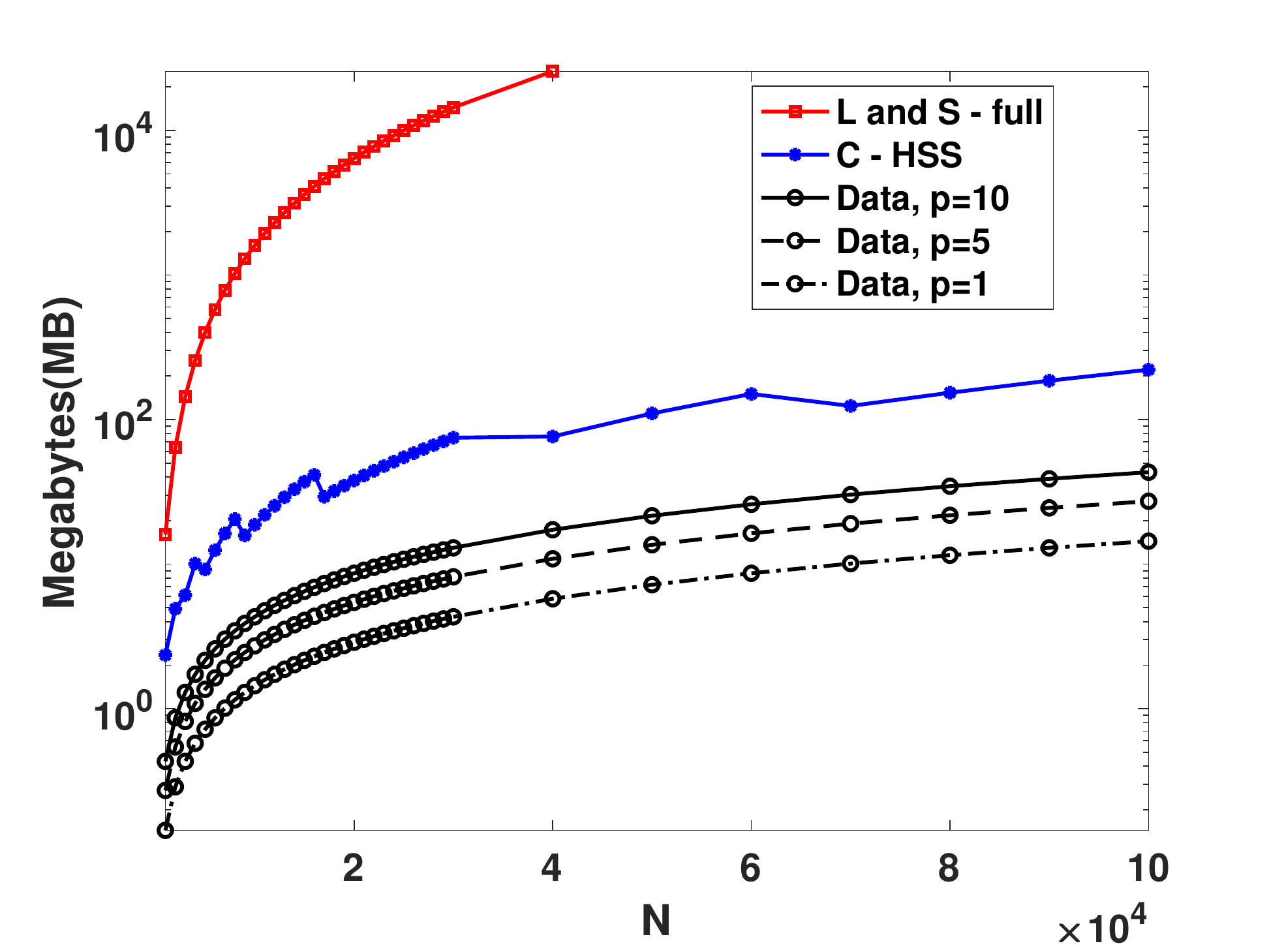}
\caption{Example~\ref{Ex.2}. Memory requirements in Megabytes to store $\LL$, $\sL$, $\widetilde\C$, and the data matrices ($\Lambda_r$, $M_r$, $\bV_r$, $\bW_r$, $\bL_r$, $\bR_r$) for different values of $N$ and $p$. }
\label{fig1.ex2.laptop}
\end{figure}

We report the results of the comparison between the different approaches in terms of run time in Table~\ref{tab1.ex2.laptop} for the number of measurements $N$ varying between $1\, 000$ to $100\, 000$, the number of inputs and outputs $p$ taking values $1$, $5$ and $10$, and the number of poles being $50$ or $100$. The ``--'' is used to indicate the instances for which we were not able to compute the reduced model~\eqref{eq:E_D0}: for $N>50\, 000$, we cannot allocate the full matrices $\LL$ and $\sL$, and for $N=30\, 000,40\, 000$ we could not compute the full SVD of $\sL-x\LL$ . Such constraints are not relevant to our proposed strategy. It is pertinent to remark the following:
\begin{enumerate}
    \item the CPU time of the full SVD approach does not depend on $p$ and $n$, only on $N$, as expected from Table \ref{tab.cost}: indeed, the cost of building $\LL$ and $\sL$ is quadratic in $N$ whereas the full SVD demands $\mathcal{O}(N^3)$ FLOPs; the full SVD approach is rarely the fastest method (it can happen for very modest values of $N$ in the considered range);
    \item the CPU time of the full assembly of $\sL-x\LL$ followed by the {\tt svds} Matlab command does not depend on $p$, only on $N$ and $n$, as expected from Table \ref{tab.cost}: the construction of $\LL$ and $\sL$ costs $\mathcal{O}(N^2)$ FLOPs, whereas the computational effort for the partial SVD depends on $n$, leading to a more demanding procedure for large $n$; it is usually the fastest approach for (very) modest values of $N$ in the considered range and $p>1$;
    \item the HSS rank of the Cauchy matrix approximation $\widetilde\C$ only depends on the frequency samples, hence on $N$ because, in our scenario, the sampling interval is the same, but the distribution of points inside the interval is different for each $N$; there may be instances when, for the same samples, the HSS rank of $\widetilde\C$ may produce slightly different results due to the randomness induced by the adaptive cross approximation procedure used in constructing $\widetilde\C$ (for instance, for $N = 50\, 000$, $p = 5$, $n = 50$ and $n = 100$, the rank is $28$, while for the rest of the values considered for $n$ and $p$, the rank is $27$); moreover, the HSS rank increases with $N$;
    \item our proposed approach is as accurate as the first two approaches, highlighting the fact that the HSS approximation $\widetilde\C$ does not lead to significant losses in the approximation properties of the reduced model~\eqref{eq:E_D0}; clearly, our approach cannot be more accurate than the traditional Loewner framework, especially when the full SVD is performed;
    \item last, but not least, the CPU time of the proposed solution depends linearly on $p$, $n$ and as $N \log N$ (Table \ref{tab.cost}), thus being the fastest method for large values of $N$; moreover, no memory constraints are present for $N$ up to $100\, 000$.
\end{enumerate}

\begin{table}[htbp!]
 \centering
  \begin{tabular}{rrr|rr|rr|rrr}
     &  &  & \multicolumn{2}{c}{Full {\tt svd}} & \multicolumn{2}{|c|}{{\tt svds} w/ $\sL-x\LL$} &\multicolumn{3}{c}{{\tt svds} w/ $\widetilde\C$} \\
$N$ & $p$ & $n$ & Time (s) & $\mathcal{H}_2$-error & Time (s) & $\mathcal{H}_2$-error & $\mathtt{hssrank}(\widetilde\C)$ & Time (s) & $\mathcal{H}_2$-error\\
    \hline
    1 000 & \multirow{10}{*}{1} & \multirow{10}{*}{50} &  
    0.28 & 3.62e-10 & \textbf{0.20} & 3.62e-10 & 15 & 1.07 & 3.62e-10 \\
    3 000 & & & 10.69 & 3.71e-10 & 3.51 & 3.71e-10 & 19 & \textbf{3.16} & 3.71e-10 \\
    5 000 & & & 20.79 & 3.7e-10 & 12.29 & 3.7e-10 & 21 & \textbf{6.99} & 3.7e-10\\
    10 000 & & & 158.41 & 3.71e-10 & 68.40 & 3.71e-10 & 22 & \textbf{15.55} & 3.71e-10 \\
    15 000 &  &  &  554.98 & 3.73e-10 & 209.31 & 3.73e-10 & 24 & \textbf{22.86} & 3.73e-10 \\
    29 000 & & & 4674.35 & 3.74e-10 & 1590.37 & 3.74e-10 & 26 & \textbf{52.33} & 3.74e-10 \\
    30 000 & & & -- & -- & 1827.45 & 3.74e-10 & 26 & \textbf{50.68} & 3.74e-10 \\
    40 000 & & & -- & -- & 11214.41 & 3.74e-10 & 27 & \textbf{71.44} & 3.74e-10 \\
    50 000 & & & -- & -- & -- & -- & 27 & \textbf{91.88} & 3.74e-10 \\
    100 000 & & & -- & -- & -- & -- & 30 & \textbf{189.71} & 3.75e-10 \\
    \hline
    1 000 & \multirow{6}{*}{1} & \multirow{6}{*}{100} & \textbf{0.21} & 9.63e-11 & 0.22 & 9.63e-11 & 15 & 1.59 & 9.64e-11 \\
    3 000 & & & 10.66 & 1.01e-10 & 5.84 & 1.01e-10 & 19 & \textbf{5.65} & 1.01e-10 \\
    5 000 & & & 20.62 & 1.01e-10 & 18.89 & 1.01e-10 & 21 & \textbf{12.77} & 1.01e-10 \\
    10 000 & & & 156.76 & 1.01e-10 & 93.16 & 1.01e-10 & 22 & \textbf{28.47} & 1.02e-10\\
    50 000 & & & -- & -- & -- & -- & 27 & \textbf{155.46} & 1.03e-10\\
    100 000 & & & -- & -- & -- & -- & 30 & \textbf{321.84} & 1.03e-10\\
    \hline 
    1 000 & \multirow{6}{*}{5} & \multirow{6}{*}{50} & 0.27 & 3.71e-10 & \textbf{0.19} & 3.71e-10 & 15 & 1.20 & 3.72e-10 \\
    3 000 & & & 10.69 & 3.17e-10 & \textbf{3.56} & 3.17e-10 & 19 & 4.59 & 3.17e-10 \\
    5 000 & & & 20.81 & 3.01e-10 & 12.29 & 3.01e-10 & 21 & \textbf{9.51} & 3.02e-10 \\
    10 000 & & & 157.75 & 2.92e-10 & 68.43 & 2.92e-10 & 22 & \textbf{19.67} & 2.92e-10\\
    50 000 & & & -- & -- & -- & -- & 28 & \textbf{107.74} & 2.88e-10\\
    100 000 & & & -- & -- & -- & -- & 30 & \textbf{230.17} & 2.88e-10\\
    \hline
    1 000 & \multirow{6}{*}{5} & \multirow{6}{*}{100} & 0.26 & 2.52e-10 & \textbf{0.25} & 2.52e-10 & 15 & 2.21 & 2.52e-10\\
    3 000 & & & 10.69 & 1.41e-10 & \textbf{5.90} & 1.41e-10 & 19 & 8.62 & 1.41e-10 \\
    5 000 & & & 20.78 & 1.35e-10 & 18.97 & 1.35e-10 & 21 & \textbf{17.73} & 1.35e-10 \\
    10 000 & & & 157.75 & 1.33e-10 & 93.58 & 1.33e-10 & 22 & \textbf{36.84} & 1.33e-10 \\
    50 000 & & & -- & -- & -- & -- & 28 & \textbf{197.72} & 1.31e-10 \\
    100 000 & & & -- & -- & -- & -- & 30 & \textbf{421.02} & 1.3e-10 \\
    \hline 
    1 000 & \multirow{6}{*}{10} & \multirow{6}{*}{50} &0.26 & 6.57e-10 & \textbf{0.18} & 6.57e-10 & 15 & 1.61 & 6.57e-10 \\
    3 000 & & & 11.05 & 3.2e-10 & \textbf{3.67} & 3.2e-10 & 19 & 5.58 & 3.2e-10 \\
    5 000 & & & 20.83 & 2.73e-10 & 12.25 & 2.73e-10 & 21 & \textbf{10.94} & 2.73e-10\\
    10 000 & & & 159.13 & 2.68e-10 & 69.34 & 2.68e-10 & 22 & \textbf{23.09} & 2.68e-10\\
    50 000 & & & -- & -- & -- & -- & 27 & \textbf{132.64} & 2.56e-10\\
    100 000 & & & -- & -- & -- & -- & 30 & \textbf{293.55} & 2.54e-10\\
    \hline
    1 000 & \multirow{6}{*}{10} & \multirow{6}{*}{100} & \textbf{0.24} & 5.27e-10 & 0.24 & 5.27e-10 & 15 & 2.88 & 5.25e-10 \\
    3 000 & & & 10.68 & 1.78e-10 & \textbf{5.94} & 1.78e-10 & 19 & 10.45 & 1.78e-10 \\
    5 000 & & & 20.84 & 1.73e-10 & \textbf{18.97} & 1.73e-10 & 21 & 20.50 & 1.73e-10 \\
    10 000 & & & 157.48 & 1.65e-10 & 93.63 & 1.65e-10 & 22 & \textbf{42.97} & 1.65e-10 \\
    50 000 & & & -- & -- & -- & -- & 27 & \textbf{248.91} & 1.58e-10 \\
    100 000 & & & -- & -- & -- & -- & 30 & \textbf{552.66} & 1.58e-10 \\
    \end{tabular}
\caption{Example \ref{Ex.2}. Computational time (in seconds) and $\mathcal{H}_2$-error achieved by each approach for different values of $N$ (number of samples), $p$ (number of inputs and outputs), and $n$ (order of the underlying system and of the model) on the employed laptop.}\label{tab1.ex2.laptop}
 \end{table}

  \begin{figure}[t]
  \centering
\begin{minipage}[c]{.45\textwidth}
\centering
\includegraphics[scale=0.5]{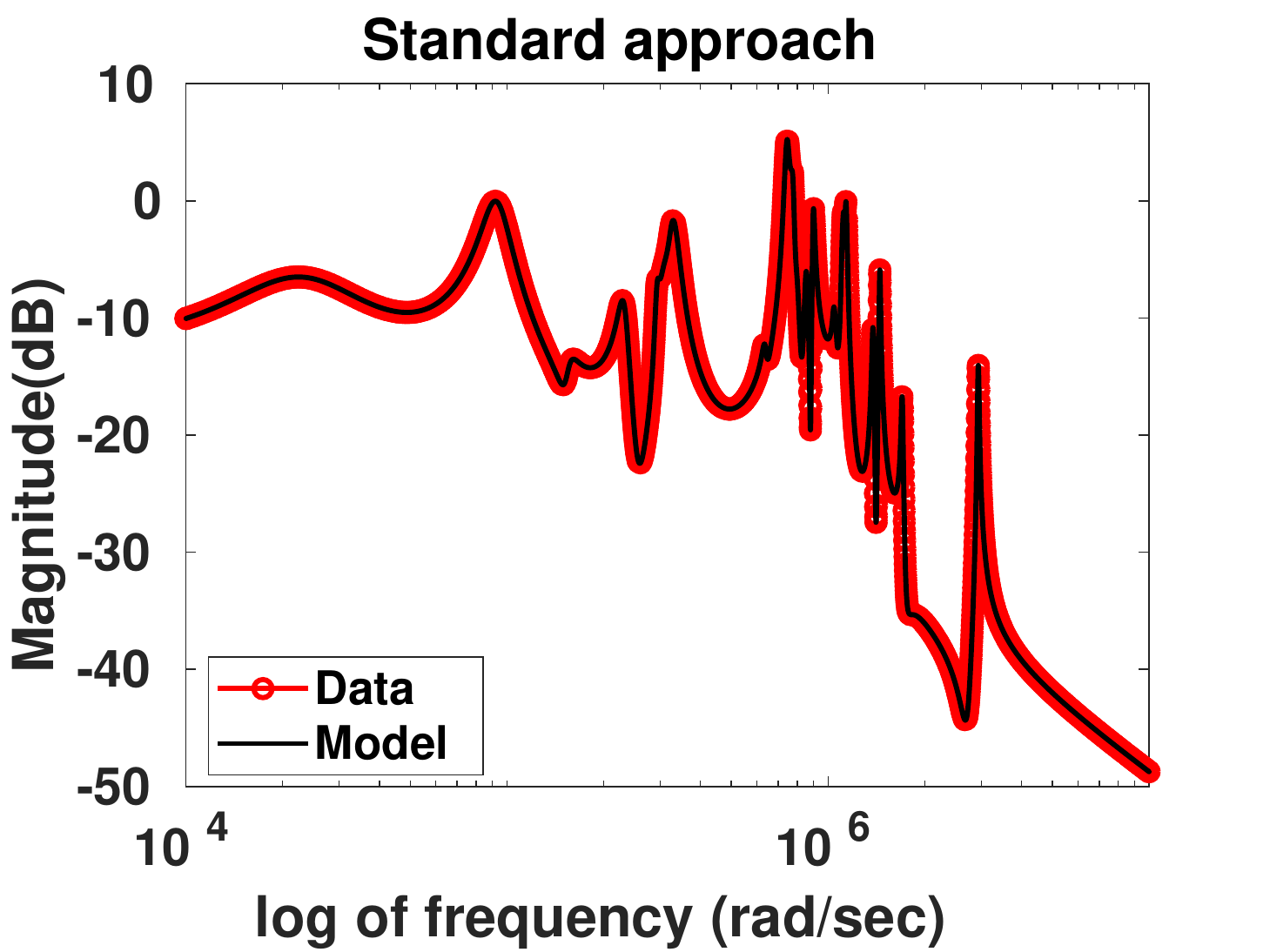}%
\end{minipage}
\hspace{6mm}%
\begin{minipage}[c]{.45\textwidth}
\centering
\includegraphics[scale=0.5]{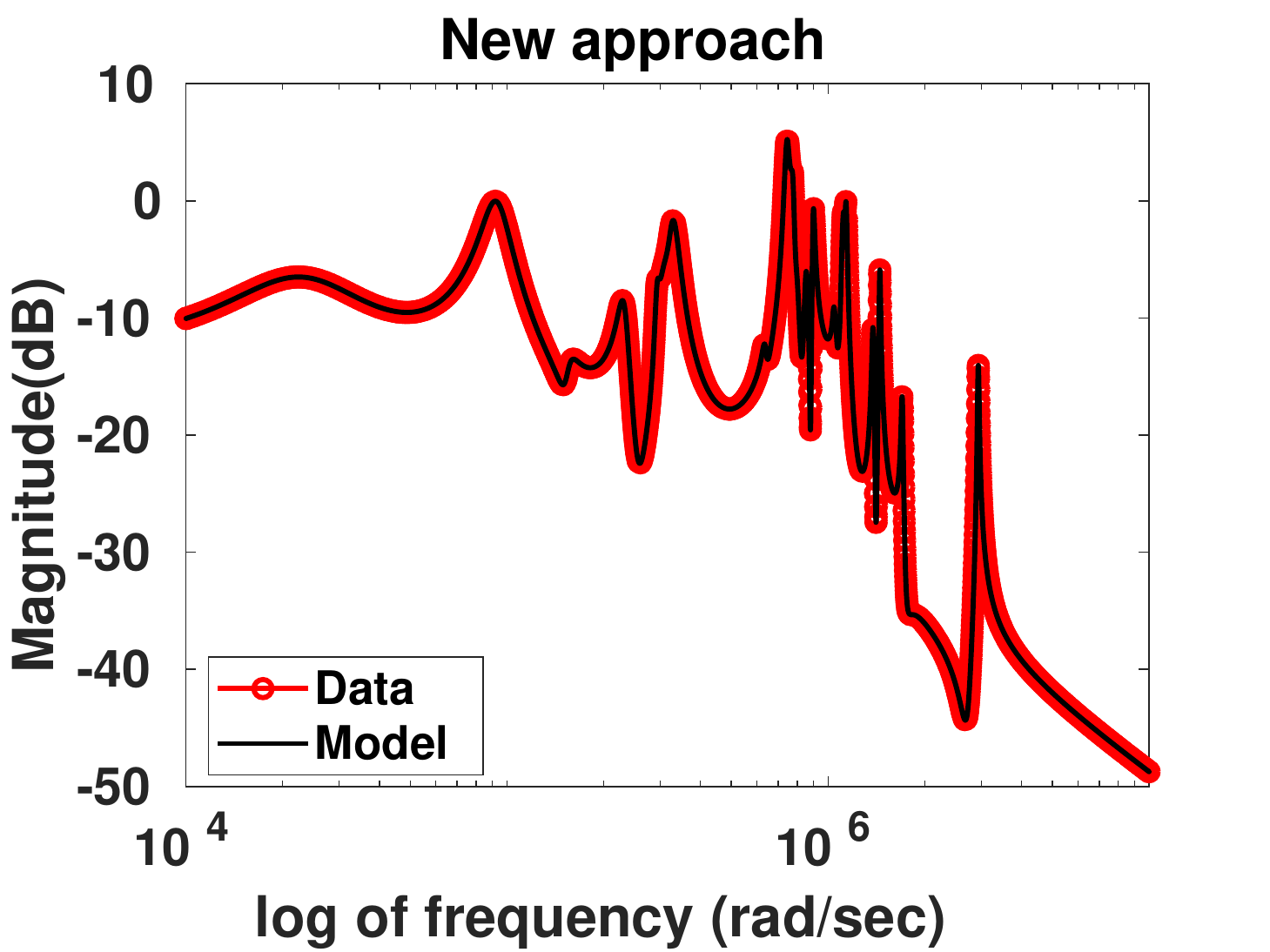}%
\end{minipage}
\caption{Example~\ref{Ex.2}. Left: Frequency response obtained with the standard approach (in black) and the measurements (in red) for $p=1$, $n=50$, and $N=10\,000$. Right: Frequency response obtained with the proposed approach (in black) and the measurements (in red) for $p=1$, $n=50$, and $N=10\,000$.}\label{fig3.ex2}
\end{figure}
 
 

In Figure~\ref{fig2.ex2.laptop} (left) we plot the computational time of the three approaches for $p=1$, $n=50$, and different values of $N$. Even though these are the same results as those reported in Table~\ref{tab1.ex2.laptop}, Figure~\ref{fig2.ex2.laptop} (left) clearly shows the $\mathcal{O}(N^3)$ trend of the full SVD scheme versus the $\mathcal{O}(N^2)$ trend of the {\tt svds} scheme versus the $\mathcal{O}(N)$ behaviour of the proposed approach. In Figure~\ref{fig2.ex2.laptop} (right) we depict, on a logarithmic scale, the running time of the proposed procedure for $n=50$ and different values of $N$ and $p$, clearly exhibiting a linear dependency on $p$ and an $N \log N$ dependency with respect to $N$.

  \begin{figure}[t]
  \centering
\begin{minipage}[c]{.45\textwidth}
\centering
\includegraphics[scale=0.4]{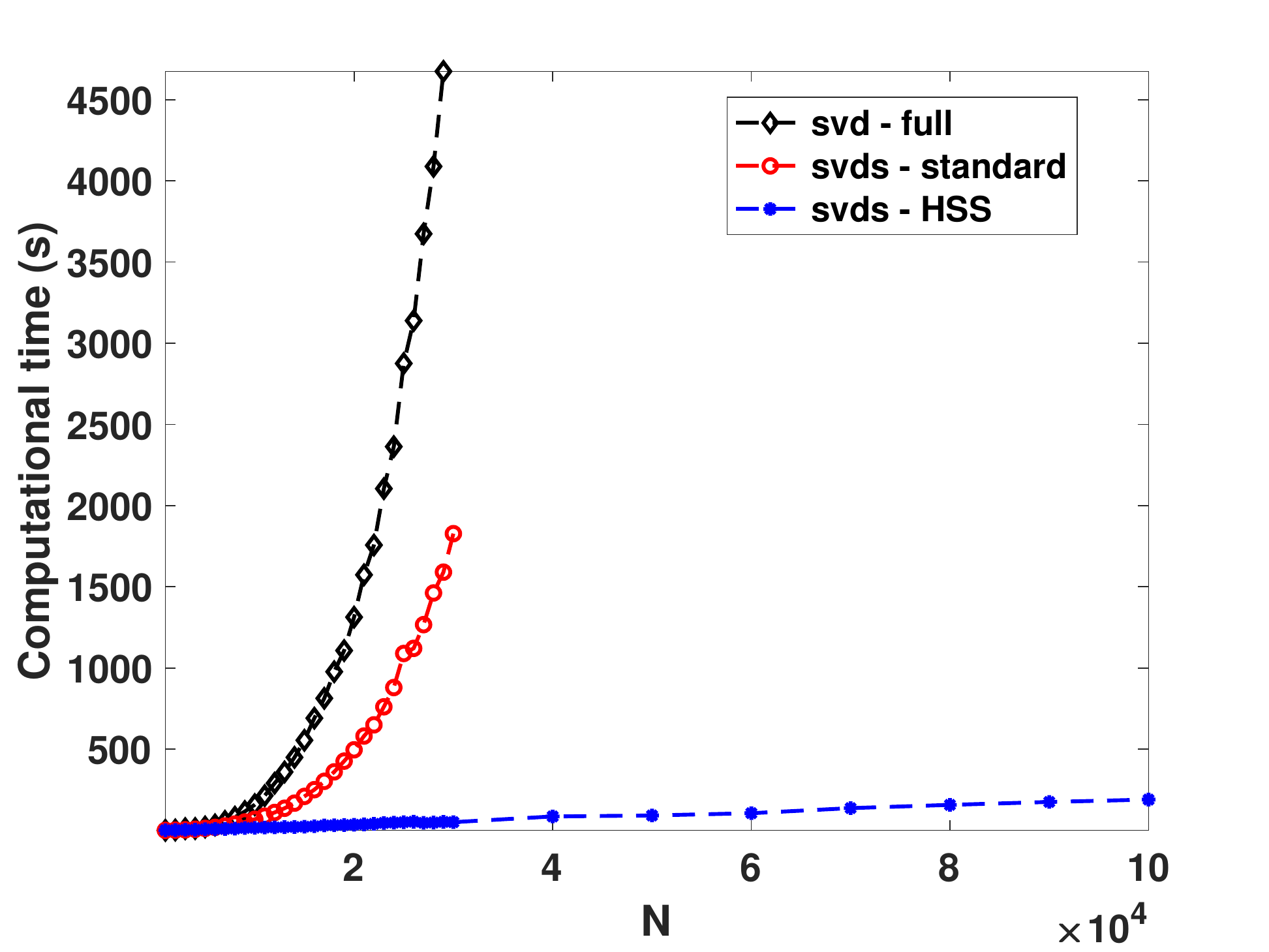}%
\end{minipage}
\hspace{8mm}%
\begin{minipage}[c]{.45\textwidth}
\centering
\includegraphics[scale=0.4]{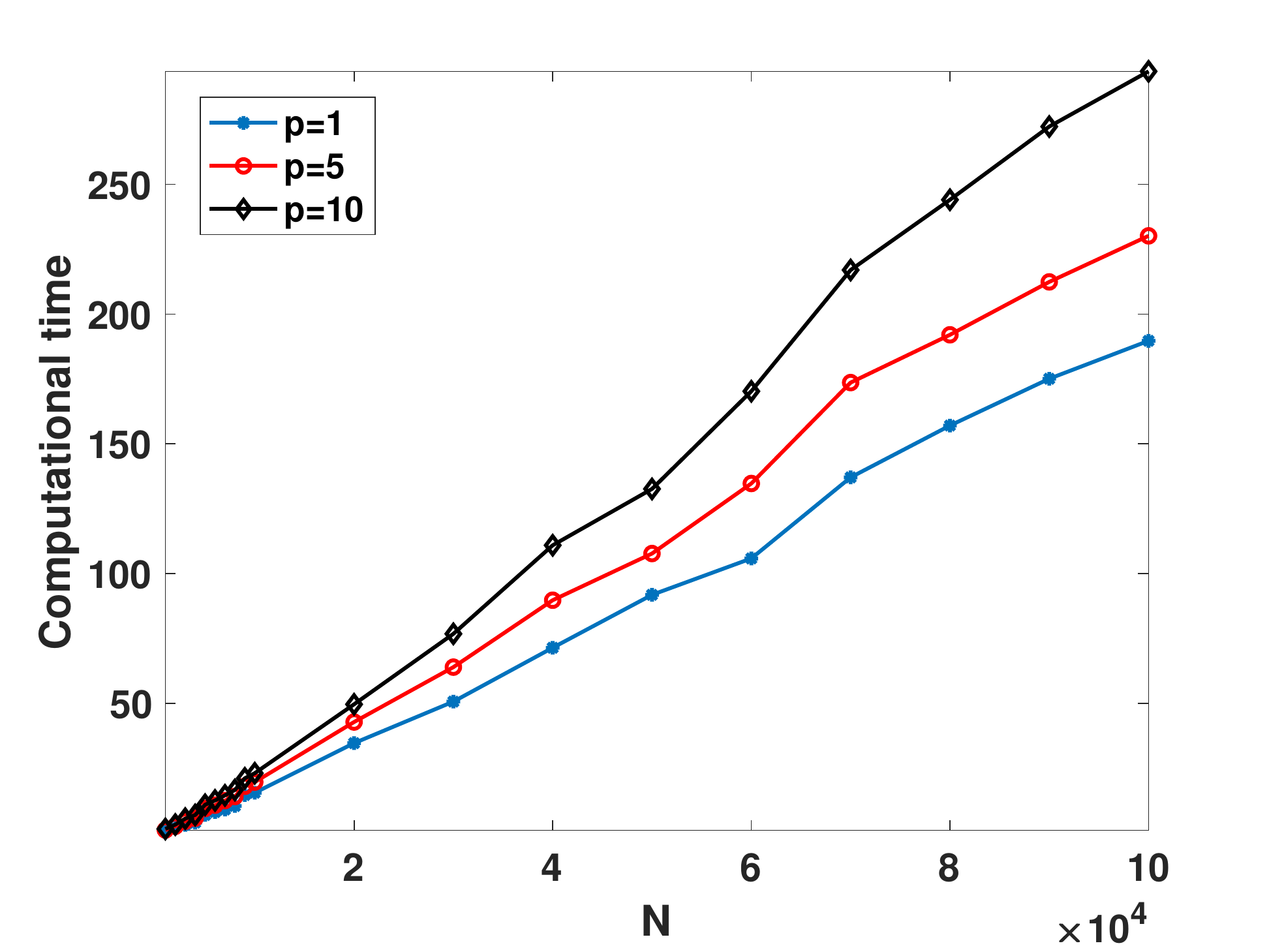}%
\end{minipage}
\caption{Example~\ref{Ex.2}. Left: Computational time achieved by the different approaches for $p=1$, $n=50$, and $N$. Right: Computational time achieved by our novel procedure for $n=50$, and different values of $N$ and $p$.}\label{fig2.ex2.laptop}
\end{figure}

 }
\end{num_example}
\begin{num_example}\label{Ex.3}
{\rm
In this example we compare the novel strategy presented in this paper to the fast Loewner SVD scheme illustrated in~\cite{FSVDL}. 
We consider the same data set as the one in Example~\ref{Ex.2}, this time with $SNR=120$ and a random $\bD\in \mathbb{R}^{p \times p}\neq \bfz$. Due to the fact that the models resulting from the Loewner framework have $\bD=\bfz$, a realization of size $n+p$ is needed to approximate the system with $\bD\neq \bfz$ \cite{artAJMACA, SLACATCAD09}. 

In \cite{FSVDL}, a Galerkin-ADI method is applied to the Sylvester equation~\eqref{eq:Sylv} satisfied by the Loewner matrix. At the $k$-th iteration, a low-rank approximation $P_k L_k Q_k^*$, $P_k,Q_k\in\mathbb{C}^{N\times\bar k}$, $L_k\in\mathbb{C}^{\bar k\times \bar k}$, to $\LL$ is thus computed. If $U_kS_kV_k^*=L_k$ denotes the SVD of $L_k$, then the matrices $P_kU_k$ and $V_k^*Q_k^*$ can be used in place of $\bX_n$ and $\bY_n$ in~\eqref{eq:E_D0} to compute the reduced model. The method is stopped whenever the norm of the residual matrix $\bm M P_k L_k Q_k^*-P_k L_k Q_k^* \bm \Lambda -\bV\bR+\bL\bW$, consisting of the left-hand side of the Sylvester equation with $\LL$ replaced by its low-rank approximation $P_k L_k Q_k^*$, is smaller than a certain threshold $\varepsilon$. In the results that follow we employ $\varepsilon=10^{-4}$, as done in~~\cite{FSVDL}. At each iteration step, the SVD of $L_k$ is truncated to keep only the $n+p$ significant values. 

We consider the {\sc Half\&Half} partition of the frequencies as this is the best scenario for the scheme coming from~\cite{FSVDL}. The {\sc Half\&Half} partition often leads to a rather fast convergence of the Galerkin-ADI method in terms of number of iterations so that a quite small approximation space is constructed. 
If different partitions were used, the Galerkin-ADI method could be equipped with a quite involved divide-and-conquer scheme; see~\cite{FSVDL}. 
On the other hand, as illustrated in Example~\ref{Ex.1}, the {\sc Half\&Half} partition leads to higher values of the HSS-rank of $\widetilde C$ than for the {\sc Even\&Odd} partition with a consequent increment in the computational efforts of our scheme. In addition, as for~\cite{FSVDL}, our tests employed complex arithmetic and did not solve the corresponding Sylvester equation \eqref{eq:Sylv_realLL} for real-coefficient matrices. 

In Table~\ref{tab1.ex3} we report the results for $p=10$, $n=50$, and different values of $N$. Notice that even though the Galerkin-ADI approach efficiently computes the approximation spaces, the construction of the reduced model~\eqref{eq:E_D0} still requires the allocation of both $\LL$ and $\sL$. Therefore, also for the Galerkin-ADI scheme severe memory constraints hold and for $N>30\, 000$, we are not able to allocate the $\LL$ and $\sL$ matrices with complex entries on the machine used for running the tests.


\begin{table}[htbp!]
 \centering
  \begin{tabular}{c|cccc|ccc}
       & \multicolumn{4}{c|}{Galerkin-ADI with $\varepsilon=10^{-4}$} & \multicolumn{3}{c}{{\tt svds} w/ $\widetilde\C$} \\
 &\# of& Scheme & Total  &  & & Total & \\
$N$ & Iter.&Time(s) & Time (s) & $\mathcal{H}_2$-error &  $\mathtt{hssrank}(\widetilde\C)$ & Time (s) & $\mathcal{H}_2$-error\\
    \hline
5 000 & 5 &2.93 & 6.61  & 1.55e-2 & 42 & 100.21 & 2.06e-9\\ 
10 000 & 5 &5.4 & 23.21  & 2.76e-2 & 46 & 174.04 & 1.81e-9\\
15 000 & 5 & 10.80 & 138.2 & 4.06e-2 & 49 & 280.11 & 1.31e-9\\
20 000 & 5 &14.67 & 342.34  & 9.45e-2 & 50 & 340.42 & 1.17e-9 \\
25 000 & 6 &23.38 & 668.56 & 1.18e-2 & 52 & 426.34 & 9.55e-10 \\
30 000 & 6 &30.67 & 1198.81 & 1.17e-1 & 52 & 540.20 & 9.06e-10 \\
\end{tabular}
\caption{Example \ref{Ex.3}. Number of iterations, computational time (in seconds) solely of the Galerkin-ADI iteration scheme together with the total time (including building the data, the full Loewner and shifted Loewner matrices and the projection step) as well as the $\mathcal{H}_2$-error achieved by the Galerkin-ADI approach. In comparison, we list the HSS-rank, the total time (in seconds) as well as the $\mathcal{H}_2$-error of the novel scheme presented in this paper for different values of $N$ (number of samples), $p=10$, and $n=50$.}\label{tab1.ex3}
 \end{table}
 
 Even though the Galerkin-ADI approach is faster for $N<20 \,000$, the computed approximation spaces are quite poor. Indeed, the computed reduced models are always $7$ orders of magnitude less accurate than the ones constructed by our approach. The paper~\cite{FSVDL} validates the Galerkin-ADI scheme on a system with randomly generated poles for various orders $n$ and number of samples $N$ but does not mention the accuracy of the resulting models. Moreover, in terms of CPU time, our results are comparable to the ones in~\cite{FSVDL} when considering the computational time solely of the Galerkin-ADI iteration, disregarding the steps involving building the full matrices and projecting these to obtain the reduced model. 
 
 The remarkable difference in the accuracy attained by the two approaches make any sort of computational comparison rather pointless. 
 However, we would like to point out that the computational time of the Galerkin-ADI approach grows quadratically with $N$ due to the need to assemble and store the full Loewner and shifted Loewner matrices, while an  $N \log N$ dependency of the computational cost of our novel approach can be evidenced once again from the timings reported in Table~\ref{tab1.ex3}. 
 
 Several ideas could be implemented to improve the accuracy of the models obtained with the Galerkin-ADI approach. In order to have the fairest comparisons with respect to our novel approach, each of these ideas will be tested separately to explore all the possibilities to enhance the Galerkin-ADI approach from~\cite{FSVDL}.
 
 First, the tolerance $\varepsilon$ for solving the Lyapunov equation via Galerkin-ADI can be chosen to a value comparable to the noise level for an $SNR$ of $120$, namely $\varepsilon =10^{-12}$. Results are detailed in Table \ref{tab1.ex3.tol} only for the case $N = 5\,000$, $p=10$, and $n=50$ as the trend is obvious from this one example. While the accuracy of the model has slightly improved with respect to results obtained for $\varepsilon=10^{-4}$, the number of iterations has also considerably increased, leading to matrices $L_k$ of much larger dimensions for which the SVD $L_k = U_kS_kV_k^*$ becomes costly. Hence, the CPU cost of the scheme has exploded and is no longer viable. In any case, even for a tolerance value close to the noise level, the accuracy of the model is several orders of magnitude worse than with our proposed technique ($10^{-3}$ versus $10^{-9}$).  
 
 \begin{table}[htbp!]
 \centering
  \begin{tabular}{c|cccc}
$\varepsilon$&\# of Iter.& Scheme Time(s)& Total Time(s) &  $\mathcal{H}_2$-error \\
    \hline
    $10^{-4}$& 5 &2.93 & 6.61 & 1.55e-2\\ 
 $10^{-12}$& 51 & 1648.02 & 1655.01 & 2.20e-3\\ 
\end{tabular}
\caption{Example \ref{Ex.3}. Number of iterations, computational time (in seconds) solely of the Galerkin-ADI iteration scheme together with the total time (including building the data, the full Loewner and shifted Loewner matrices and the projection step) as well as the $\mathcal{H}_2$-error achieved by the Galerkin-ADI approach for $N = 5\,000$, $p=10$, and $n=50$.}\label{tab1.ex3.tol}
 \end{table}
 
Second, it is always advisable to compute the projection subspaces from a linear combination of $\sL$ and $\LL$, namely $\sL-x\LL$ rather than only $\LL$, as the Loewner matrix $\LL$ encodes the strictly rational part and the addition of $\sL$ provides all the information on the system, including its polynomial part (the $\bD$-term). 
We apply the low-rank Galerkin-ADI method to the Sylvester equation fulfilled by $\sL-x\LL$ thus computing a matrix $P_kZ_kQ_k^*$ such that $P_kZ_kQ_k^*\approx \sL-x\LL$.
Results are detailed in Table \ref{tab1.ex3.lincomb} for the case $\varepsilon =10^{-4}$, $N = 5\,000$, $p=10$, and $n=50$. For all instances considered, results were comparable in terms of CPU time to those obtained when considering solely the Sylvester equation satisfied by $\LL$ in the Galerkin-ADI iteration (listed in the first line of Table~\ref{tab1.ex3.lincomb} for reference), while in terms of accuracy, they are slightly worse. For this example, the sole benefit of using a linear combination $\sL-x\LL$ might be the system identification properties as, in principle, a sharp drop in the SVD of $Z_k$ reveals the degree of the underlying system.
 
 \begin{table}[htbp!]
 \centering
  \begin{tabular}{l|cccc}
&\# of Iter.& Scheme Time(s)& Total Time(s) &  $\mathcal{H}_2$-error \\
    \hline
$\LL$& 5 &2.93 & 6.61  & 1.55e-2\\ 
$\sL$& 6 & 3.96 & 6.57 & 5.22e-2\\ 
$\sL-x\LL$, $x = f(1)$& 4 & 2.97 & 9.51 & 9.23e-2\\ 
$\sL-x\LL$, $x = f(N/2)$& 4 & 2.89 & 9.34 & 9.23e-2\\ 
$\sL-x\LL$, $x = f(N)$& 4 & 2.95 & 9.55 & 9.23e-2\\ 
\end{tabular}
\caption{Example \ref{Ex.3}. Number of iterations, computational time (in seconds) solely of the Galerkin-ADI iteration scheme together with the total time (including building the data, the full Loewner and shifted Loewner matrices and the projection step) as well as the $\mathcal{H}_2$-error achieved by the Galerkin-ADI approach on the Sylvester equations satisfied by $\LL$, $\sL$ and $\sL-x\LL$, for $\varepsilon =10^{-4}$, $N = 5\,000$, $p=10$, and $n=50$.}\label{tab1.ex3.lincomb}
 \end{table}
 
The third avenue worth exploring is employing real arithmetic and the corresponding Sylvester equations \eqref{eq:Sylv_realLL} and \eqref{eq:Sylv_realsL}. Table \ref{tab1.ex3.real} shows the results obtained using real arithmetic, both for the Galerkin-ADI scheme, as well as our proposed method. For reference, the first line in Table \ref{tab1.ex3.real} lists the results previously obtained in complex arithmetic. For the method in \cite{FSVDL}, the cost of the scheme has mostly increased, due to more complicated Sylvester equations in \eqref{eq:Sylv_realLL} and \eqref{eq:Sylv_realsL}. The CPU cost of building the data matrices, the full Loewner and shifted Loewner matrices has also increased, yielding a total cost far superior to that obtained in complex arithmetic. In some instances, the accuracy has improved slightly. On the other hand, the real arithmetic causes the HSS-rank of the Cauchy matrix approximation to be much smaller with a remarkable impact on the CPU time and almost no effects on the model accuracy when using our novel approach. 
 
\begin{table}[htbp!]
 \centering
  \begin{tabular}{c|cccc|ccc}
       & \multicolumn{4}{c|}{Galerkin-ADI with $\varepsilon=10^{-4}$} & \multicolumn{3}{c}{{\tt svds} w/ $\widetilde\C$} \\
 &\# of& Scheme & Total  &  & & Total & \\
$N$ & Iter.&Time(s) & Time (s) & $\mathcal{H}_2$-error &  $\mathtt{hssrank}(\widetilde\C)$ & Time (s) & $\mathcal{H}_2$-error\\
    \hline
$\LL$ complex & 5 &2.93 & 6.61 & 1.55e-2 & 42 & 100.21 & 2.06e-9\\ 
$\LL$ & 3 &1.76 & 50.43 & 1.03e-2 & 24& 81.56 & 2.05e-9\\
$\sL$& 6 & 3.61 & 52.42 & 2.66e-3 & 24 & 80.97 & 2.05e-9\\
$\sL-x\LL$, $x = f(1)$& 7 & 18.57 & 65.84 & 2.53e-3 & 24 & 80.84 & 2.05e-9 \\
$\sL-x\LL$, $x = f(N/2)$& 5 & 15.13 & 63.63 & 1.87e-2& 24 & 82.11 & 2.05e-9 \\
$\sL-x\LL$, $x = f(N)$& 6 & 16.47& 66.72 & 9.01e-2 & 24 & 82.95 & 2.05e-9 \\
\end{tabular}
\caption{Example \ref{Ex.3}. Number of iterations, computational time (in seconds) solely of the Galerkin-ADI iteration scheme together with the total time (including building the data, the full Loewner and shifted Loewner matrices and the projection step) as well as the $\mathcal{H}_2$-error achieved by the Galerkin-ADI approach. In comparison, we list the HSS-rank, the total time (in seconds) as well as the $\mathcal{H}_2$-error of the novel scheme presented in this paper for different values of $N$ (number of samples), $p=10$, and $n=50$ when employing real arithmetic.}\label{tab1.ex3.real}
 \end{table}
 
 
 We conclude this example by mentioning that the use of a \emph{hybrid} approach may be fruitful. In particular, our novel approach can be employed to avoid storing the large and dense Loewner and shifted Loewner matrices. Then, the Galerkin-ADI scheme can be used to  compute the first dominant singular vectors of $\sL-x\LL$, instead of employing {\tt svds}, thus also being able to identify the order of the underlying system. However, the accuracy will not be comparable to that of our proposed approach. We implemented this idea and list the CPU times of the various steps in Table \ref{tab1.ex3.hyb} together with the resulting accuracy for Galerkin-ADI applied to solving the Sylvester equation \eqref{eq:Sylv_realLL} for $\LL$ in real arithmetic with $\varepsilon=10^{-4}$ for $N = 5\,000$, $p=10$, and $n=50$. Plots of the responses of our proposed approach, together with the Galerkin-ADI scheme as proposed in \cite{FSVDL} and the hybrid approach are shown in Figure~\ref{fig.ex3.resp}. Even though the general shape of the response is well captured, some resonances are not modeled accurately, as expected from the much higher model errors reported earlier. This can be noticed better from the error plots in Figure~ \ref{fig.ex3.err}. 
 
  \begin{table}[htbp!]
 \centering
  \begin{tabular}{c|c|c|c|c}
Data matrices Time(s)&Galerkin-ADI Scheme Time(s) & Projection Time(s) & Total Time(s) &  $\mathcal{H}_2$-error \\
    \hline
 0.8& 1.76 &4.65 & 7.21 & 2.1e-2
\end{tabular}
\caption{Example \ref{Ex.3}. Computational time (in seconds) of the three individual steps in the \emph{hybrid} approach: setting up of the data matrices, the Galerkin-ADI iteration scheme and projection to obtain the reduced model, together with the total time as well as the $\mathcal{H}_2$-error for $N = 5000$, $p=10$, and $n=50$ in real arithmetic.}\label{tab1.ex3.hyb}
 \end{table}
 
  \begin{figure}[t]
  \centering
  \hspace*{-1.2cm}
\begin{minipage}[c]{.32\textwidth}
\centering
\includegraphics[scale=0.3]{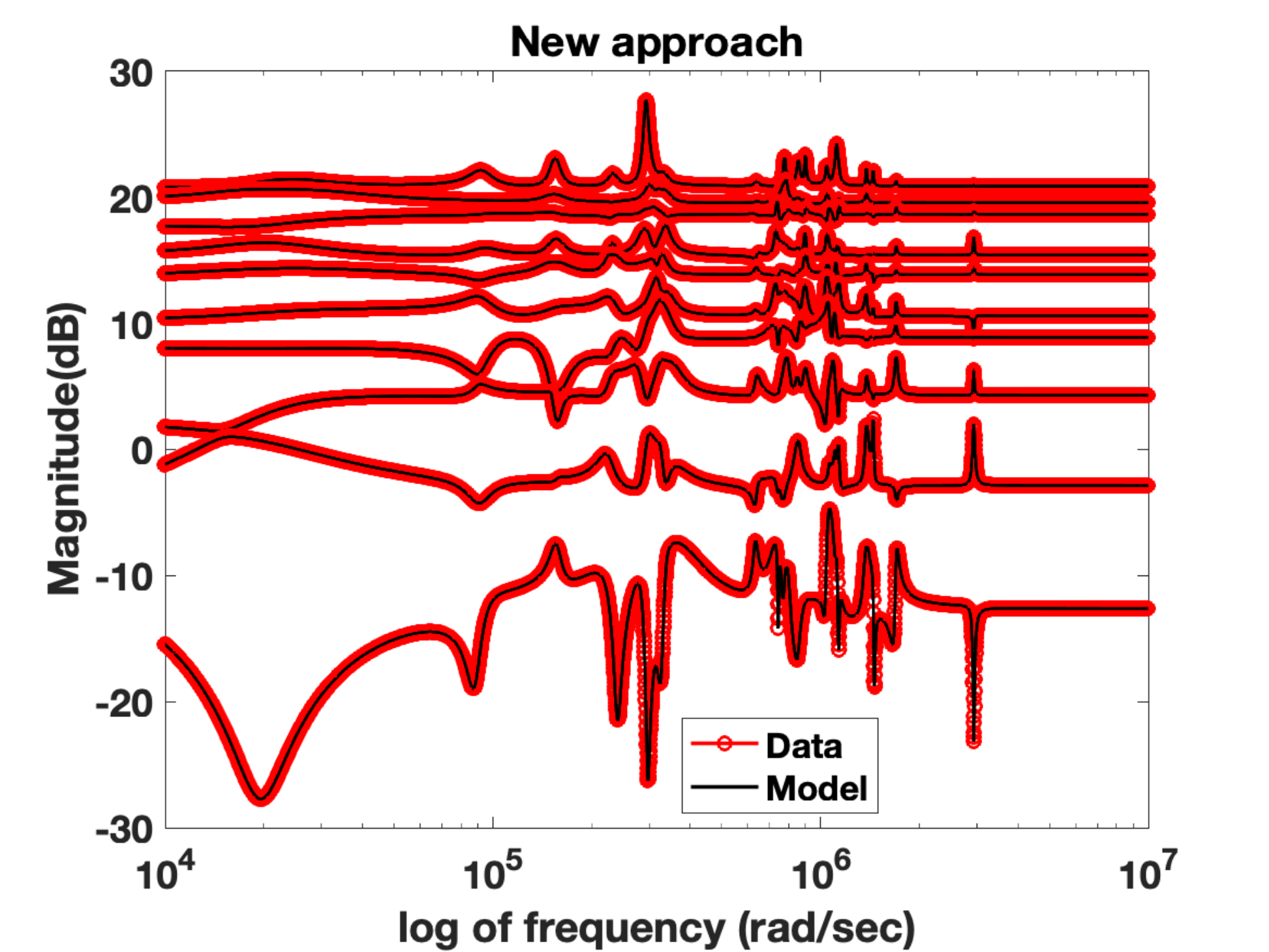}%
\end{minipage}
\hspace{3mm}%
\begin{minipage}[c]{.32\textwidth}
\centering
\includegraphics[scale=0.3]{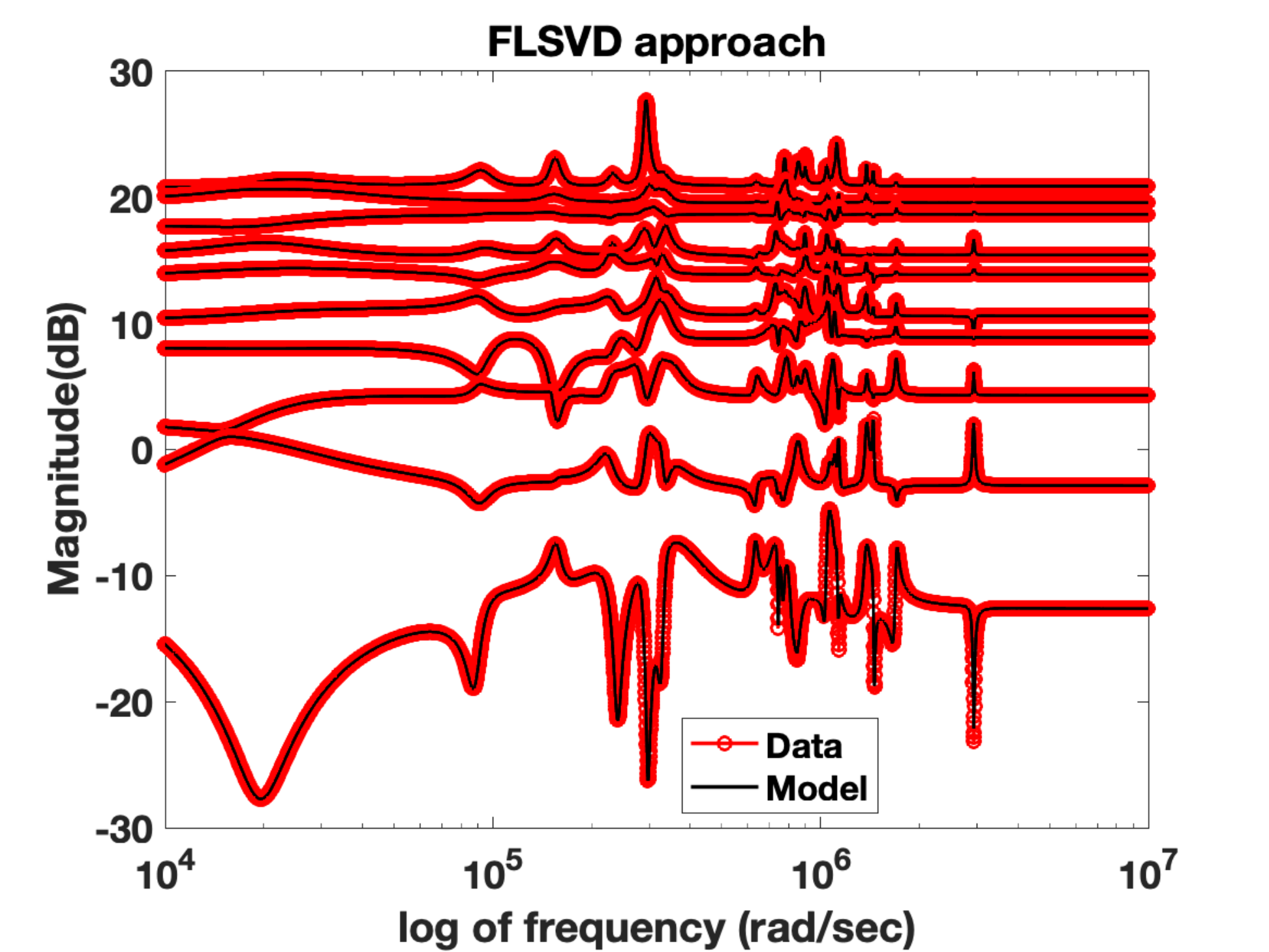}%
\end{minipage}
\hspace{3mm}%
\begin{minipage}[c]{.32\textwidth}
\centering
\includegraphics[scale=0.3]{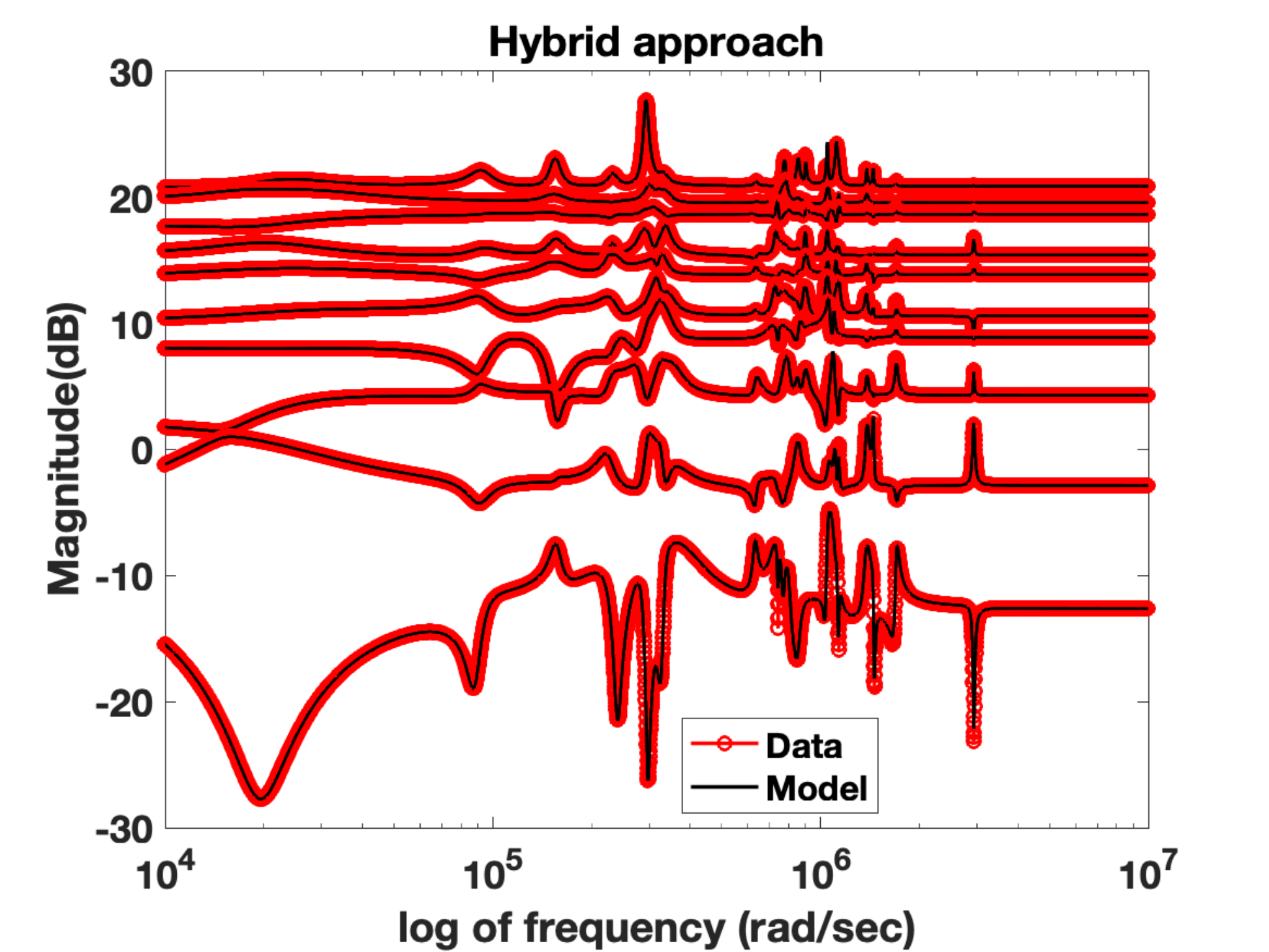}%
\end{minipage}
\hspace*{-0.2cm}
\caption{Example~\ref{Ex.3}. Frequency response of the model (in black) and the measurements (in red) for $N = 5\,000$, $p=10$, and $n=50$ using our proposed approach, Galerkin-ADI as in \cite{FSVDL} and the hybrid approach, employing real arithmetic. }\label{fig.ex3.resp}
\end{figure}

  \begin{figure}[t]
  \centering
  \hspace*{-1.2cm}
\begin{minipage}[c]{.32\textwidth}
\centering
\includegraphics[scale=0.3]{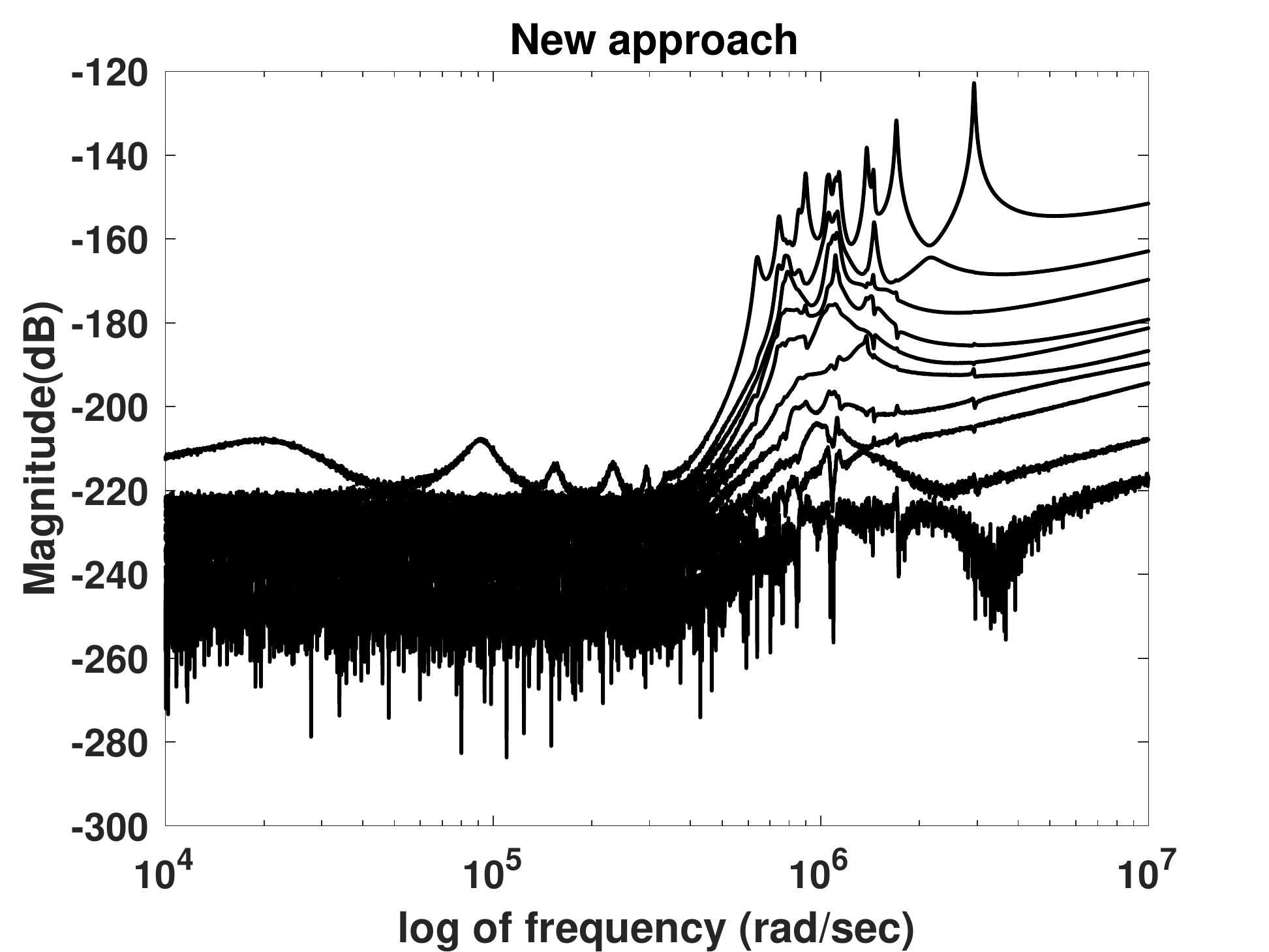}%
\end{minipage}
\hspace{3mm}%
\begin{minipage}[c]{.32\textwidth}
\centering
\includegraphics[scale=0.3]{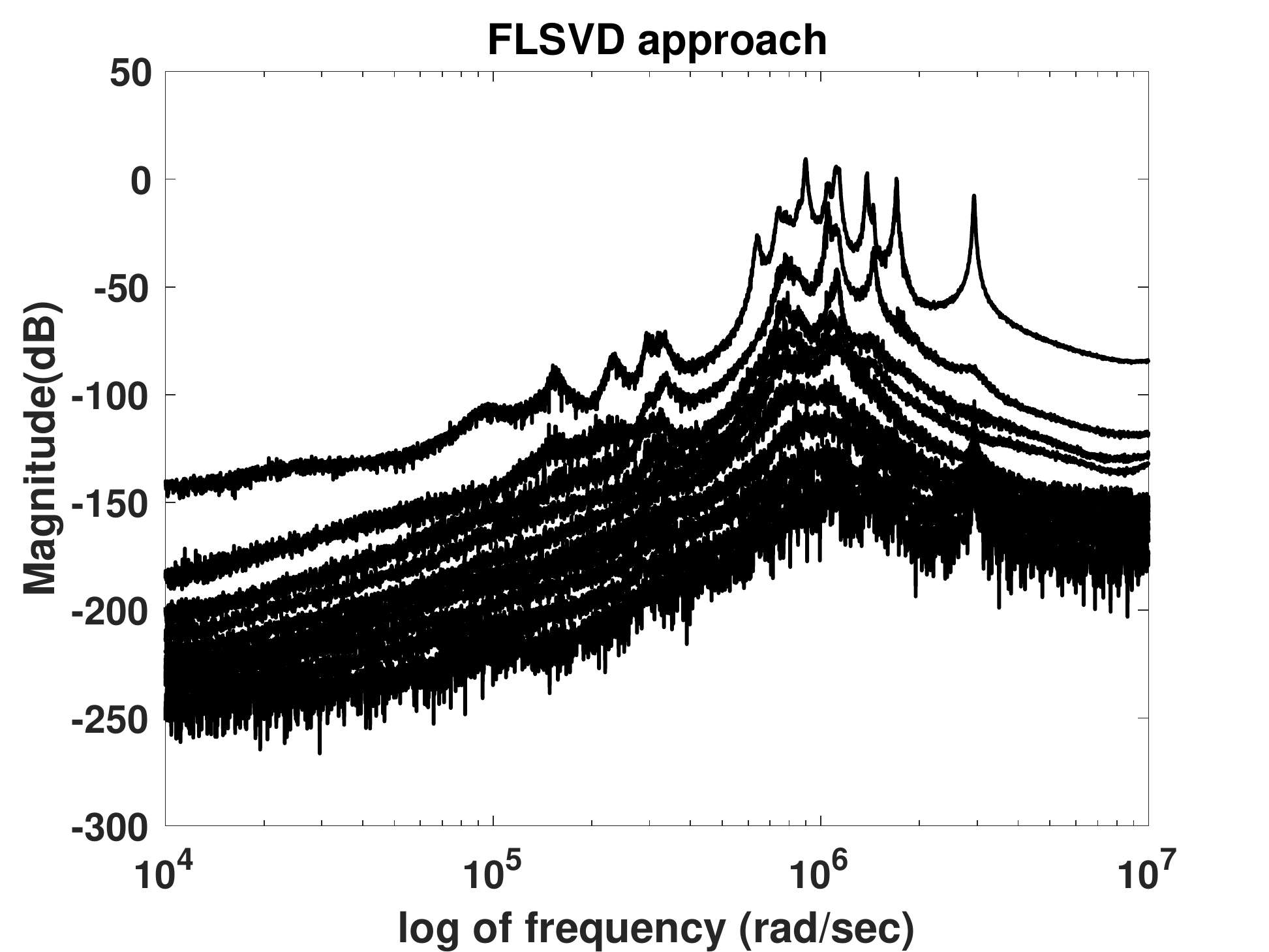}%
\end{minipage}
\hspace{3mm}%
\begin{minipage}[c]{.32\textwidth}
\centering
\includegraphics[scale=0.3]{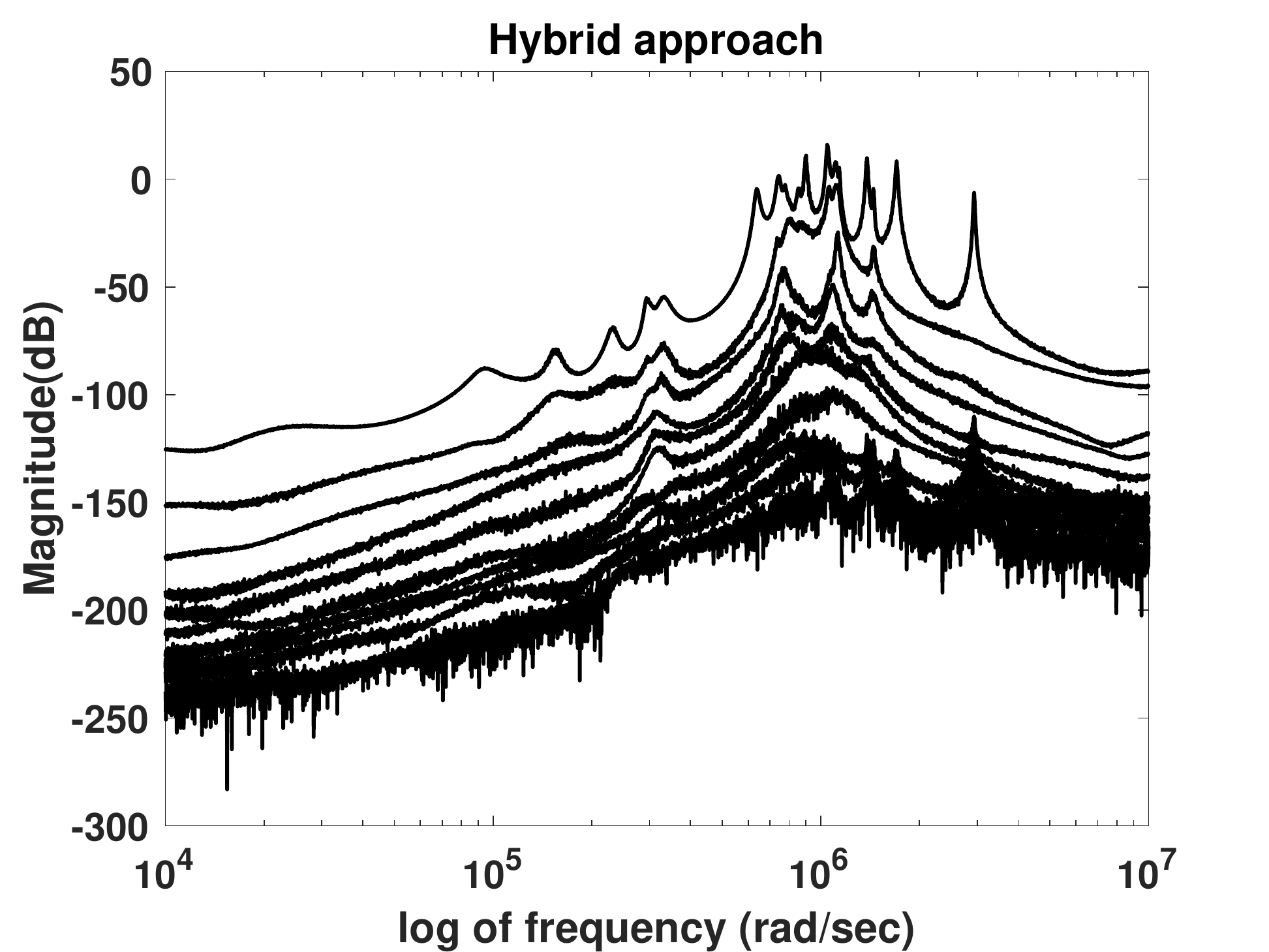}%
\end{minipage}
\hspace*{-0.2cm}
\caption{Example~\ref{Ex.3}. Error plots for $N = 5\,000$, $p=10$, and $n=50$ using our proposed approach, Galerkin-ADI as in \cite{FSVDL} and the hybrid approach, employing real arithmetic. }\label{fig.ex3.err}
\end{figure}
}
\end{num_example}

\section{Conclusion}\label{sect:concl}
By exploiting the Cauchy-like structure of the Loewner and shifted Loewner matrices, a novel strategy for reducing the computational costs and the memory requirements of the Loewner framework has been proposed. In particular, the use of the HSS-format leads to tremendous savings in the storage demand and computational efforts of the overall scheme. Indeed, except for the construction of $\widetilde \C$ whose cost is polylogarithmic in $N$, both the memory requirements and the computational cost of iteratively performing the SVD now linearly depend on the cardinality of the considered data set.

The success of our procedure strongly relies on the capability of representing the Cauchy matrix $\C$ in terms of an HSS-matrix $\widetilde\C$ with low $\left(\alpha,\beta\right)$ rank of the off-diagonal blocks. Even though we restricted ourselves to showing how different, but common, partitions of the frequencies affect the HSS-rank of $\widetilde\C$, a thorough analysis of their connection may be beneficial. Moreover, we have always computed $\widetilde\C$ at high accuracy. We believe that the employment of more inexact, and thus with a lower rank, HSS-representations of $\C$ and its effects on the accuracy of the overall scheme may be another interesting research direction which is worth pursuing depending on the application at hand. 

The strategy presented in this paper can be applied to more sophisticated problems as long as the Loewner and shifted Loewner matrices maintain a Cauchy-like structure. In particular, our approach can be employed with minor modifications in model order reduction of parametrized \cite{IonitaPhD}, linear switched \cite{GosPA18}, and bilinear systems \cite{AntGI16}.

\section*{Acknowledgements}
We are in debt with Leonardo Robol for some help with \cite{hmtoolbox} and fruitful discussions about the topic of this paper. His assistance is greatly appreciated. We also thank Peter Benner and Jens Saak for insightful comments on earlier versions of the manuscript.

The first author is member of the Italian INdAM Research group GNCS.

 \section*{Declarations}
 The research presented in this paper is based upon work supported by the National Science Foundation under Grant No. DMS-1439786 while both the authors were in residence at the Institute for Computational and Experimental Research in Mathematics (ICERM) in Providence, RI, during the \emph{Model and Dimension Reduction in Uncertain and Dynamic Systems} program.
Even though the second half of the program had to be performed virtually due to the restrictions caused by the COVID-19 pandemic, 
we are extremely grateful to the organizers of the program and the whole staff of ICERM for doing whatever possible to maintain an exciting, fruitful, and high-quality working environment.

 The authors have no conflicts of interest to declare that are relevant to the content of this article.
 
 The datasets and algorithms generated during and/or analysed during the current study are available from the corresponding author on reasonable request. Moreover, the approach presented in this paper will be included in the {\tt hm-toolbox} in the near future.

\bibliographystyle{siam}
\bibliography{biblio}
\end{document}